\documentclass[10pt]{article}

\usepackage{geometry}  % set the margins to 1in on all sides
\usepackage[english]{babel}
\usepackage[utf8]{inputenc}
\usepackage[T1]{fontenc}

\usepackage{hyperref}
\usepackage{xcolor}

\usepackage{amsmath}
\usepackage{amssymb}
\usepackage{amsthm}
\usepackage{bbm}
\usepackage{graphicx}
\usepackage{algorithm}
\usepackage{algpseudocode}
\RequirePackage[authoryear]{natbib}
\newtheorem{theorem}{Theorem}[section]
\newtheorem{proposition}[theorem]{Proposition}
\newtheorem{lemma}[theorem]{Lemma}

\theoremstyle{definition}
\newtheorem{definition}[theorem]{Definition}

\newtheorem*{proof sketch}{Proof sketch}

\title{Algorithmic stability implies training-conditional coverage\\ for distribution-free prediction methods}
\author{Ruiting Liang\thanks{Committee on Computational and Applied Mathematics, University of Chicago} \ and Rina Foygel Barber\thanks{Department of Statistics, University of Chicago}}
\date{\today}

\begin{document}

\maketitle

\begin{abstract}
In a supervised learning problem, given a predicted value that is the output of some trained model, how can we quantify our uncertainty around this prediction?
Distribution-free predictive inference aims to construct prediction intervals around this output, with valid coverage that does not rely on assumptions on the distribution of the data or the nature of the model training algorithm. 
Existing methods in this area, including conformal prediction and jackknife+, offer theoretical guarantees that hold marginally (i.e., on average over a draw of training and test data). 
In contrast, training-conditional coverage is a stronger notion of validity that ensures predictive coverage of the test point for most draws of the training data, and is thus a more desirable property in practice. Training-conditional coverage was shown by \citet{vovk2012conditional} to hold for the split conformal method, but recent work by \citet{bian2023training} proves that such validity guarantees are not possible for the full conformal and jackknife+ methods without further assumptions. In this paper, we show that an assumption of algorithmic stability ensures that the training-conditional coverage property holds for the full conformal and jackknife+ methods.
\end{abstract}
%------------------

\section{Introduction}\label{sec:intro}
The field of distribution-free predictive inference has gained popularity among the modern machine learning community, since it can be paired with any model-fitting algorithm to provide a finite-sample guarantee. Given training data $(X_1,Y_1),\dots,(X_n,Y_n)\in\mathcal{X}\times\mathbb{R}$ (used to train a regression model) and a test feature value $X_{n+1}\in\mathcal{X}$ (for which we would like to predict the unobserved response $Y_{n+1}\in\mathbb{R}$), the aim of a distribution-free method is to construct a prediction interval $\widehat{C}_n(X_{n+1})$ satisfying
the coverage guarantee
\begin{equation}\label{eqn:marginal}
    \mathbb{P} \left\{Y_{n+1} \in \widehat{C}_{n} (X_{n+1}) \right\} \ge 1-\alpha,
\end{equation}
where the only assumption is that the training and test data $(X_1,Y_1),\dots,(X_{n+1},Y_{n+1})$ are i.i.d.\ draws from some distribution $P$---but, no assumptions are placed on $P$, or on the model fitting algorithm. However, this type of guarantee is \emph{marginal}, meaning that the probability of coverage is computed on average over a draw of the training data and test point. In particular, it would be more valuable to have a coverage guarantee conditional on the training data, i.e.,
\[\mathbb{P} \left\{Y_{n+1} \in \widehat{C}_{n} (X_{n+1}) \mid \{(X_i,Y_i)\}_{i=1,\dots,n} \right\} \gtrapprox 1-\alpha,\]
which would imply that the \emph{observed} prediction band $\widehat{C}_n$ will have $\approx 1-\alpha$ coverage over future test data, i.e., with respect to the random draw of the test point $(X_{n+1},Y_{n+1})$. 

To make this goal more precise, define the training data set\footnote{ Here and throughout the paper, we will write ``data set'' to denote a collection of data points, even though the same data point may potentially appear multiple times---that is, more formally, we might refer to $\mathcal{D}_n$ as a \emph{multiset} of data points.} as $\mathcal{D}_n = \{(X_i,Y_i)\}_{i=1,\dots,n}$, and define the conditional noncoverage rate as
\[\alpha_P(\mathcal{D}_n) = \mathbb{P}\left\{Y_{n+1}\not\in \widehat{C}_n(X_{n+1})\mid \mathcal{D}_n\right\},\]
the probability that the prediction band $\widehat{C}_n$ (which itself is a function of the training data set $\mathcal{D}_n$) will fail to cover for a new test point $(X_{n+1},Y_{n+1})$ drawn from $P$. Then the \emph{marginal} coverage guarantee~\eqref{eqn:marginal}
is equivalent to the bound
$\mathbb{E}[\alpha_P(\mathcal{D}_n)] \leq \alpha$.
Of course, this bound does not ensure that $\alpha_P(\mathcal{D}_n)$ will be $\lessapprox\alpha$ for a typical draw of the training data $\mathcal{D}_n$; indeed, in a degenerate setting, we might have the worst-case scenario
\begin{equation}\label{eqn:worst_case}
\alpha_P(\mathcal{D}_n)\approx 1\textnormal{ with probability $\approx \alpha$}; \quad
\alpha_P(\mathcal{D}_n)\approx 0\textnormal{ with probability $\approx 1- \alpha$},\end{equation}
which does not contradict the marginal coverage guarantee, but is clearly highly undesirable from a practical perspective since it means we have a substantial probability of getting a completely unreliable prediction interval.
To avoid this worst-case scenario, we can instead consider a training-conditional coverage guarantee, i.e., a bound of the form
\begin{equation}\label{eqn:trainingcc}
    \mathbb{P} \left\{ \alpha_{P}(\mathcal{D}_{n}) > \alpha + \mathrm{o}(1) \right\} = \mathrm{o}(1).
\end{equation}
This type of guarantee would ensure that we have avoided the worst-case scenario described above.

\subsection{Prior results}
Prior work has considered the question of training-conditional coverage for a range of commonly used distribution-free predictive inference methods: split conformal prediction, full conformal prediction, and jackknife+. We will give background on each of these methods shortly in Section~\ref{sec:background}. \citet{vovk2012conditional} establishes a training-conditional coverage bound of the type~\eqref{eqn:trainingcc} for split conformal prediction. In contrast, \citet{bian2023training} prove that for the full conformal prediction and jackknife+ methods, no such bound can be established without further assumptions; their work constructs specific counterexamples where the worst-case scenario outlined in~\eqref{eqn:worst_case} can occur for these two methods. Moreover, as demonstrated empirically by \citet[Section 5]{bian2023training}, both the jackknife+ and full conformal method can fail to have training-conditional coverage even in a simple linear regression setting, if the sample size and dimension are approximately equal (in which case the problem becomes ill-conditioned, leading to instability). Hence, \citet{bian2023training} conjecture that an assumption of \emph{algorithmic stability} may be sufficient to ensure training-conditional coverage for these methods. 

For the jackknife, previous work has shown that stability leads to training-conditional coverage in an asymptotic sense \citep{steinberger2023conditional}; under additional assumptions, \citet{amann2023assumption} also show finite-sample guarantees for jackknife, which we will discuss more later on in Section~\ref{sec:prior_jack+}.

\subsection{Our contributions}
In this work, we resolve the open question posed by \citet{bian2023training}, and verify that algorithmic stability can indeed enable a training-conditional coverage guarantee of the form~\eqref{eqn:trainingcc} for both jackknife+ and full conformal. Along the way, we will need to examine algorithmic stability in a different way. To describe this informally, let $\widehat{\mu}:\mathcal{X}\rightarrow\mathbb{R}$ be a  regression model fitted to a training set of size $n$, and let $\widehat{\mu}'$ be the model fitted to a perturbed version of the same training set, with $m$ many data points removed or replaced, where $m$ is large (but may be much smaller than $n$). We will see that it is sufficient to assume $\widehat{\mu}(X)\approx \widehat{\mu}'(X)$, i.e., the predictions returned by our fitted model are essentially unchanged when we perturb $m$ (out of $n$) training points: our main results establish that this type of stability assumption ensures a training-conditional coverage guarantee of the type~\eqref{eqn:trainingcc} for both full conformal prediction and jackknife+.

\subsection{Organization of paper}
The remainder of the paper is organized as follows. In Section~\ref{sec:background}, we give additional details on the distribution-free predictive inference methods mentioned above (split conformal, full conformal, and jackknife+), and give some brief background on algorithmic stability. In Section~\ref{sec:main_results}, we introduce our new stability framework and give our main results establishing training-conditional coverage guarantees. Sections~\ref{sec:stability_frameworks} and~\ref{sec:1-vs-m} examine our stability assumptions in more detail, comparing to existing frameworks, building equivalences between different notions of stability and providing examples of stable algorithms. In Section~\ref{sec:proofs} we present the proofs of our main results. We conclude with a short discussion in  Section~\ref{sec:discussion}.

\section{Background}\label{sec:background}
In this section, we review some background on distribution-free prediction methods, and compare the marginal and training-conditional coverage guarantees offered by various common methods. We also give a short overview of the literature on algorithmic stability. 

\subsection{Distribution-free prediction}

Given training data $(X_1,Y_1),\dots,(X_n,Y_n)$ drawn i.i.d.\ from an \emph{unknown} distribution $P$ on $\mathcal{X}\times\mathbb{R}$, we would like to perform inference on predictions for future data points, without relying on (potentially incorrect) assumptions on $P$. At the same time, we want to leverage the power of existing models and algorithms that might provide accurate predictions in many practical settings. For example, if we run a kernel regression on the training data to compute an estimate of the conditional mean function $\widehat{\mu}(x) \approx \mathbb{E}[Y\mid X=x]$, we would like to then compute a prediction interval around this estimate (e.g, $\widehat{\mu}(x)$ plus or minus some margin of error) that will be narrower if $\widehat{\mu}$ is accurate (say, if $P$ satisfies certain smoothness assumptions), but whose validity does not rely on any assumptions.
Before proceeding, we first introduce one piece of notation: throughout this paper, $\mathcal{A}$ will denote a regression algorithm,\footnote{This framework, and all results in this paper, extend in a straightforward way to the case of a randomized regression algorithm (e.g., stochastic gradient descent), but for clarity of the exposition we will only discuss the nonrandom case throughout.} which inputs a training data set of any size (i.e., a set of $(X,Y)$ pairs), and returns a fitted regression function $\widehat{\mu}:\mathcal{X}\rightarrow\mathbb{R}$ that is measurable:
\[\mathcal{A}: \bigcup_{n\geq 0} (\mathcal{X}\times\mathbb{R})^n \rightarrow \{\textnormal{measurable functions $\mathcal{X}\rightarrow\mathbb{R}$}\}.\] From this point on, we will assume without further comment that $\mathcal{A}$ is symmetric in the training data: for each $n\geq 1$ and any permutation $\pi$ on $[n]=\{1,\dots,n\}$, it  must hold that
\begin{equation*}
    \mathcal{A} ((X_{1},Y_{1}),...,(X_{n},Y_{n})) = \mathcal{A}((X_{\pi(1)},Y_{\pi(1)}),...,(X_{\pi(n)},Y_{\pi(n)}))
\end{equation*}
for any data set $(X_1,Y_1),\dots,(X_n,Y_n)$.

With the problem setting in place, we now briefly review three methods for distribution-free prediction---split conformal, full conformal, and the jackknife+.  For further background on these methods, including questions of how the choice of the base algorithm $\mathcal{A}$ affects the resulting prediction interval in terms of length or other properties, and how the resulting prediction intervals compare with more traditional model-based methods, see, e.g., \citet{vovk2005algorithmic,lei2018distribution,sesia2020comparison}.

\subsubsection{Split conformal prediction}
Split conformal prediction \citep{vovk2005algorithmic,papadopoulos2002inductive} begins by partitioning the $n$ available labeled data points into a training set of size $n_0$, $\{(X_i,Y_i)\}_{i=1,\dots,n_0}$, and a holdout or calibration set of size $n_1 = n-n_0$, $\{(X_i,Y_i)\}_{i=n_0+1,\dots,n}$.
Let 
\[\widehat{\mu}_{n_0} = \mathcal{A}\big((X_1,Y_1),\dots,(X_{n_0},Y_{n_0})\big)\]
denote the fitted model when the regression algorithm $\mathcal{A}$ is applied to the training set. The split conformal prediction interval is then given by
\[\widehat{C}^{\textnormal{split}}_n(X_{n+1}) = \widehat{\mu}_{n_0}(X_{n+1})\pm \textnormal{Quantile}\big(\{|Y_i - \widehat{\mu}_{n_0}(X_i)|\}_{i=n_0+1,\dots,n}; (1-\alpha)(1+1/n_1)\big),\]
where the margin of error is essentially the $(1-\alpha)$ quantile of the residuals $|Y_i - \widehat{\mu}_{n_0}(X_i)|$ on the holdout set $i\in\{n_0+1,\dots,n\}$, with a small correction to adjust for the finite sample size.\footnote{Here for split conformal (and for the methods below) we are giving the most basic version of the construction, which produces prediction intervals of the form $\widehat{\mu}_{n_0}(X_{n+1})$
plus or minus a margin of error; alternative implementations of the method can allow for more rich and complex forms of the prediction interval, for example, to account for nonconstant variance in the noise---see, e.g., \citet{lei2018distribution,romano2019conformalized}.}

Split conformal prediction is well known to offer both marginal coverage \citep{vovk2005algorithmic,papadopoulos2002inductive} and training-conditional coverage \citep{vovk2012conditional}---specifically, for any algorithm $\mathcal{A}$ and for data drawn i.i.d.\ from any distribution $P$, the training-conditional noncoverage rate $\alpha^{\textnormal{split}}_P(\mathcal{D}_n) = \mathbb{P}\{Y_{n+1}\not\in\widehat{C}_n^{\textnormal{split}}(X_{n+1})\mid \mathcal{D}_n\}$ satisfies
\[\mathbb{E}[\alpha^{\textnormal{split}}_P(\mathcal{D}_n)]\leq \alpha, \textnormal{ and }\mathbb{P}\left\{ \alpha_P^{\textnormal{split}}(\mathcal{D}_n) > \alpha  + \sqrt{\frac{\log(1/\delta)}{2n_1}}\right\} \leq \delta\]
for any $\delta\in(0,1)$. However, split conformal suffers from the drawback of data splitting: the fitted model $\widehat{\mu}_{n_0}$ will likely be less accurate since it is trained on $n_0$ many, rather than $n$ many, data points (e.g., in practice, $n_0 = n/2$ might be a common choice).

\subsubsection{Full conformal prediction}
In contrast to split conformal, the full conformal prediction method \citep{vovk2005algorithmic} avoids the loss of statistical accuracy incurred by data splitting, at the price of computational cost. The method is defined via the prediction set
\[\widehat{C}^{\textnormal{CP}}_n(X_{n+1}) = \left\{y \in\mathbb{R} : |y - \widehat{\mu}^y(X_{n+1})| \leq \textnormal{Quantile}\big(\{|Y_i - \widehat{\mu}^y(X_i)|\}_{i\in[n]}; (1-\alpha)(1+1/n)\big)\right\}, \]
where for each $y\in\mathbb{R}$, 
\[\widehat{\mu}^y = \mathcal{A}\big((X_1,Y_1),\dots,(X_n,Y_n),(X_{n+1},y)\big)\]
is the fitted model when the algorithm $\mathcal{A}$ is run on the training data together with hypothesized test point $(X_{n+1},y)$.  Here, we refer to $\widehat{C}^{\textnormal{CP}}_n(X_{n+1})$ as a prediction set, rather than prediction interval, since in principle the construction may return a set that is not a connected interval; in practice, however, it is often the case that $\widehat{C}^{\textnormal{CP}}_n(X_{n+1})$ is indeed an interval.

Full conformal satisfies marginal coverage, i.e., $\mathbb{E}[\alpha^{\textnormal{CP}}_P(\mathcal{D}_n)]\leq \alpha$ (where as before we define $\alpha^{\textnormal{CP}}_P(\mathcal{D}_n) = \mathbb{P}\{Y_{n+1}\not\in\widehat{C}_n^{\textnormal{CP}}(X_{n+1})\mid \mathcal{D}_n\}$), again with no assumptions on the unknown distribution $P$ \citep{vovk2005algorithmic}---this is the same as split conformal prediction (and indeed, split conformal can be viewed as a special case of full conformal, thus inheriting its marginal coverage guarantee).
However, the two methods are quite different in terms of training-conditional coverage. Indeed, \citet{bian2023training} show that, unlike for split conformal, no assumption-free training-conditional coverage guarantee is possible for full conformal, by constructing an explicit counterexample of an algorithm $\mathcal{A}$ and a distribution $P$ for which it holds that
\begin{equation}\label{eqn:worst_case_bian}
    \mathbb{P} \left\{\alpha^{\textnormal{CP}}_{P}(\mathcal{D}_{n}) \ge 1-n^{-2} \right\} \ge \alpha - 6\sqrt{\frac{\log n}{n}},
\end{equation}
i.e., the example essentially attains the worst-case scenario described in~\eqref{eqn:worst_case}.  \citet[Section 5]{bian2023training} also demonstrate empirically that full conformal can fail to achieve the training-conditional coverage even with a simple linear regression algorithm, in an ill-conditioned regime.

\subsubsection{Jackknife+}
The jackknife+ \citep{barber2021predictive} is an alternative method that offers a compromise between the statistical efficiency of full conformal and the computational efficiency of split conformal; it is closely related to the cross-conformal procedure \citep{vovk2015cross}. 
For each $i \in [n]$, define
\[
    \widehat{\mu}_{-i} = \mathcal{A}\left((X_{1},Y_{1}),...,(X_{i-1},Y_{i-1}), (X_{i+1},Y_{i+1}),..., (X_{n},Y_{n}) \right) ,
\]
which fits the model to the training data with $(X_i,Y_i)$ removed, and let $R_i = |Y_i - \widehat{\mu}_{-i}(X_i)|$ be the $i$th leave-one-out residual. 
The jackknife+ prediction interval is given by
\begin{multline*}\widehat{C}^{\textnormal{J+}}_n(X_{n+1}) = \Big[-\textnormal{Quantile}\big(\{-\widehat{\mu}_{-i}(X_{n+1}) + R_i \}_{i\in[n]};(1-\alpha)(1+1/n)\big) , \\\textnormal{Quantile}\big(\{\widehat{\mu}_{-i}(X_{n+1}) + R_i \}_{i\in[n]};(1-\alpha)(1+1/n)\big)\Big].
\end{multline*}
The jackknife+ method is known to offer a weaker marginal distribution-free coverage guarantee \citep{barber2021predictive}:  the miscoverage level may be inflated by a factor of $2$, i.e., $\mathbb{E}[\alpha^{\textnormal{J+}}_P(\mathcal{D}_n)]\leq 2\alpha$ (where $\alpha^{\textnormal{J+}}_P(\mathcal{D}_n) = \mathbb{P}\{Y_{n+1}\not\in\widehat{C}_n^{\textnormal{J+}}(X_{n+1})\mid \mathcal{D}_n\}$) is the best assumption-free guarantee. 
Moreover, as for full conformal, the jackknife+ method does not achieve assumption-free training conditional coverage: \citet{bian2023training} construct an explicit example of an algorithm $\mathcal{A}$ and distribution $P$ for which the same worst-case result~\eqref{eqn:worst_case_bian} holds for jackknife+  (and, as for full conformal, also show a simulation where jackknife+  fails to achieve training-conditional coverage in an ill-conditioned linear regression setting).

\subsection{Algorithmic stability}

At a high level, algorithmic stability describes the following property of a modeling algorithm: slight perturbations of the input (i.e., the training data) should have only mild effects on the output (i.e., the predictions returned by the fitted model).
Algorithmic stability as an important property of learning algorithms has been intensively studied in the literature. It has been established that certain notions of stability is essential to derive nontrivial generalization error bounds; related works include \citet{rogers1978finite, devroye1979distribution, devroye1979distribution2, kearns1997algorithmic, bousquet2002stability}. Connections have also been made between stability and effective model selection as in \citet{meinshausen2010stability, shah2013variable, lange2002stability}.

In the statistical literature, there are a range of definitions of stability that have been studied. A common choice is an assumption of the form\footnote{Here $\mathrm{o}_P(1)$ denotes ``little-o in probability'' notation: a sequence of random variables $A_1,A_2,\dots$ satisfies $A_n = \mathrm{o}_P(1)$ if $\mathbb{P}\{|A_n| > \epsilon\}\rightarrow 0$ as $n \to \infty$, for any $\epsilon>0$.}
\begin{equation}\label{eqn:stability_intro}\left|\widehat{\mu}(X_{n+1}) - \widehat{\mu}_{-i}(X_{n+1})\right| = \mathrm{o}_P(1),\end{equation}
where $\widehat{\mu} = \mathcal{A}\big(\{(X_j,Y_j)\}_{j\in[n]}\big)$, while as before we define
$\widehat{\mu}_{-i}= \mathcal{A}\big(\{(X_j,Y_j)\}_{j\in[n]\backslash \{i\}}\big)$---note that $\widehat{\mu}$ and $\widehat{\mu}_{-i}$ are fitted on the same training set except for one data point. Alternatively, we might make this assumption in expectation,
\[\mathbb{E}\left[\left|\widehat{\mu}(X_{n+1}) - \widehat{\mu}_{-i}(X_{n+1})\right|\right]= \mathrm{o}(1). \]This type of assumption is referred to as ``hypothesis stability'' in the work of \citet{kearns1997algorithmic,bousquet2002stability}.

Some algorithms, such as ridge regression and bagging \citep{buhlmann2002analyzing, grandvalet2004bagging, elisseeff2005stability,soloff2024stability,soloff2023bagging}, are known to satisfy algorithmic stability properties due to their construction. However, most modern modeling algorithms are too complex to be analyzed theoretically, i.e., determining whether a stability property is guaranteed to hold is difficult. Recent work by \citet{kim2023black} establishes that testing stability empirically is in fact impossible---any test that provides an assumption-free ``certificate'' of stability cannot have power that is much better than random.

Another line of work that connects stability and inference is closely related to the field of differential privacy, which is typically referred to as \emph{adaptive data analysis}. There, a notion of stability is extracted from differential privacy and is used to construct valid inferences when hypotheses are determined in a data-driven way \citep{dwork2015generalization, dwork2015preserving, rogers2016max}.  \citet{zrnic2023post} build on this idea to derive a framework for post-selection inference using stability achieved by randomization. We note that the notions of stability these works considered are different from ours, and typically much stronger: the privacy literature typically considers uniform notions of stability, i.e., assumptions of the flavor $\max_{i\in[n]}\left|\widehat{\mu}(X_{n+1}) - \widehat{\mu}_{-i}(X_{n+1})\right| = \mathrm{o}_P(1)$, which requires that \emph{no training point} is allowed to be overly influential, while weaker notions of stability only require that \emph{a randomly selected training point} should not be overly influential.

\subsubsection{Connection to distribution-free prediction}\label{sec:stability_connection_to_DF}
Interestingly, under algorithmic stability assumptions, the factor of 2 can be essentially removed from the marginal coverage guarantee for the jackknife+: \citet{barber2021predictive} show that an inflated version of the jackknife+ prediction interval, defined as
\begin{multline}\label{eqn:J+_inflate}\widehat{C}^{\textnormal{$\gamma$-J+}}_n(X_{n+1}) = \Big[-\textnormal{Quantile}\big(\{-\widehat{\mu}_{-i}(X_{n+1}) + R_i \}_{i\in[n]};(1-\alpha)(1+1/n)\big) - \gamma, \\\textnormal{Quantile}\big(\{\widehat{\mu}_{-i}(X_{n+1}) + R_i \}_{i\in[n]};(1-\alpha)(1+1/n)\big) + \gamma\Big]
\end{multline}
for some small inflation parameter $\gamma$,
satisfies a stronger marginal coverage guarantee, under an algorithmic stability assumption. Specifically, a marginal coverage result of the form $\mathbb{E}[\alpha_P^{\textnormal{$\gamma$-J+}}(\mathcal{D}_n)]\leq \alpha + \mathrm{o}(1)$ is achieved under assumptions of the type~\eqref{eqn:stability_intro}.  Relatedly, earlier work by~\citet{steinberger2023conditional} establish asymptotic guarantees for both marginal and training-conditional coverage for the jackknife (rather than jackknife+, i.e., the classical version of leave-one-out cross-validation), under a similar type of algorithmic stability assumption  together with an assumption that the conditional distribution $P_{Y|X}$ has a Lebesgue density. (See also Section~\ref{sec:prior_jack+} below for a discussion of the more recent work of \citet{amann2023assumption} on this question.)

Based on these findings, \citet{bian2023training} conjectured that a stability assumption may be sufficient to ensure training-conditional coverage for jackknife+ and for full conformal prediction---this is the starting point of our present work.

\section{Main results}\label{sec:main_results}
We now turn to our main results, establishing the connection between algorithmic stability and training-conditional coverage for both jackknife+ and full conformal. 

\subsection{Results for jackknife+}
We begin by setting up some notations for quantifying algorithmic stability. The following definition seeks to answer the question: when we run a regression algorithm $\mathcal{A}$ on a training set and use it for prediction on a new test point, how sensitive is the output to a slight perturbation of the training set? 
\begin{definition}[Out-of-sample $m$-stability]\label{def:stability_E_out}
    Given an algorithm $\mathcal{A}$, a distribution $P$ on $\mathcal{X}\times\mathbb{R}$, and any $n,m\geq 0$, we define the \emph{out-of-sample $m$-stability} as
    \[\beta^{\textnormal{out}}_{m,n}(\mathcal{A},P) = \mathbb{E}\left[\big|\widehat{\mu}_n(X)- \widehat{\mu}_{n+m}(X)\big|\right],\]
    where the fitted models are defined as
    \[\widehat{\mu}_n = \mathcal{A}\big((X_1,Y_1),\dots,(X_n,Y_n)\big), \ \ \widehat{\mu}_{n+m} = \mathcal{A}\big((X_1,Y_1),\dots,(X_{n+m},Y_{n+m})\big),\]
    and where the expected value is taken with respect to $(X_1,Y_1),\dots,(X_{n+m},Y_{n+m}),(X,Y)\stackrel{\textnormal{iid}}{\sim}P$.
\end{definition}

We are now ready to state our main result regarding training-conditional coverage for the jackknife+.

\begin{theorem}[Training-conditional coverage for jackknife+]\label{thm:jack_main}
Fix any distribution $P$ on $\mathcal{X}\times\mathbb{R}$, any symmetric algorithm $\mathcal{A}$, any sample size $n\geq 1$, and any inflation parameter $\gamma>0$. Let
\[\alpha_P^{\textnormal{$\gamma$-J+}}(\mathcal{D}_n) = \mathbb{P}\left\{ Y_{n+1}\not\in \widehat{C}_n^{\textnormal{$\gamma$-J+}}(X_{n+1})\mid \mathcal{D}_n\right\}\]
denote the training-conditional miscoverage rate of the $\gamma$-inflated jackknife+ prediction interval $\widehat{C}^{\textnormal{$\gamma$-J+}}_n(X_{n+1})$ defined in~\eqref{eqn:J+_inflate}. 
Then, for any $m\geq 1$ and any $\delta \in (0,1)$,
\[\mathbb{P}\left\{ \alpha_P^{\textnormal{$\gamma$-J+}}(\mathcal{D}_n) < \alpha + 3\sqrt{\frac{\log(1/\delta)}{2\min\{n,m\}}} + 2\sqrt[3]{\frac{2\beta_{m,n-1}^{\textnormal{out}}(\mathcal{A},P)}{\gamma}} \right\} \geq 1 - 3\delta - \sqrt[3]{\frac{2\beta_{m,n-1}^{\textnormal{out}}(\mathcal{A},P)}{\gamma}}\]
where probability is taken with respect to $\mathcal{D}_n = \big\{(X_i,Y_i)\big\}_{i\in[n]}\stackrel{\textnormal{iid}}{\sim} P$, and where $\beta^{\textnormal{out}}_{m,n-1}(\mathcal{A},P)$ is defined as in Definition~\ref{def:stability_E_out}.
\end{theorem}

To help parse this result, recall that our goal is to show that $\alpha_P(\mathcal{D}_n)$ is not much larger than $\alpha$, with high probability. For this theorem to imply a conclusion of this type, we need to have large sample size $n$ (as indeed is the case even for training-conditional coverage of split conformal), and we also need $\beta^{\textnormal{out}}_{m,n-1}(\mathcal{A},P)$ to be small for some large choice of $m$, i.e., the 
$\mathcal{A}$ must satisfy out-of-sample $m$-stability for some large $m$.

The result in Theorem~\ref{thm:jack_main} above only gives a coverage guarantee for an \emph{inflated} jackknife+ interval---that is, if we take the original prediction interval (or equivalently, choose the inflation parameter to be $\gamma=0$), we do not obtain a meaningful guarantee of training-conditional coverage. However, by adding an additional assumption, we can obtain a guarantee for the \emph{uninflated} interval as well: intuitively, if $Y\mid X$ has a continuous distribution, then the slightly inflated interval should have approximately the same coverage as the original interval itself.

\begin{theorem}[Training-conditional coverage for jackknife+ without inflation]\label{thm:jack_main_uninflated}
    Under the setting and notation of Theorem~\ref{thm:jack_main}, assume also that the conditional distribution of $Y\mid X$ has density $f_{Y|X}(y\mid x)$ with respect to Lebesgue measure, and define
    \[B_{\textnormal{dens}} = \mathbb{E}\left[\sup_{y\in\mathbb{R}}f_{Y|X}(y\mid X)\right],\]
    where the expected value is taken with respect to the marginal distribution of $X$ under $P$. Then, for any $m\geq 1$ and any $\delta\in(0,1)$,
    \[\mathbb{P}\left\{ \alpha_P^{\textnormal{J+}}(\mathcal{D}_n) < \alpha + 3\sqrt{\frac{\log(1/\delta)}{2\min\{n,m\}}} + 4\sqrt[4]{2B_{\textnormal{dens}}\cdot \beta_{m,n-1}^{\textnormal{out}}(\mathcal{A},P)}  \right\} \geq 1 - 3\delta - \sqrt[4]{2B_{\textnormal{dens}}\cdot \beta_{m,n-1}^{\textnormal{out}}(\mathcal{A},P)}.
    \]
\end{theorem}

\subsubsection{Preview of the proof technique}
The full proof of Theorem~\ref{thm:jack_main} is fairly technical, and we defer it to Section~\ref{sec:proofs} (along with the proofs of our other main results). However, to give intuition for the role of $m$-stability, we give a short preview here to highlight the underlying ideas. The main strategy is to consider a larger data set, $(X_1,Y_1),\dots,(X_{n+m},Y_{n+m})$. Then taking the first $n$ points, $\mathcal{D}_n = \big((X_1,Y_1),\dots,(X_n,Y_n)\big)$, as the training set, the training-conditional noncoverage $\alpha_P(\mathcal{D}_n)$ can be approximated as the \emph{empirical} noncoverage on the test set $\{(X_i,Y_i)\}_{i=n+1,\dots,n+m}$---this is the reason that we need to take a large value of $m$, i.e., why $1$-stability is not sufficient. The proof then proceeds by comparing the $i$th leave-one-out model, trained on data points indexed by $[n]\backslash\{i\}$, against an ``oracle'' method that trains on a larger set of $n+m-1$ data points given by $[n+m]\backslash\{i\}$; we use the assumption of $m$-stability to ensure that these two models provide similar predictions.

\subsection{Results for full conformal}
In order to extend our results to full conformal prediction, we need to modify our notion of stability. The next definition seeks to quantify whether the prediction returned by an algorithm $\mathcal{A}$ is sensitive to a slight perturbation of the training set, if we are returning a prediction \emph{for a data point in the training set itself}.
\begin{definition}[In-sample $m$-stability]\label{def:stability_E_in}
    Given an algorithm $\mathcal{A}$, a distribution $P$ on $\mathcal{X}\times\mathbb{R}$, and any $n\geq 1$, $m\geq 0$,  we define the \emph{in-sample $m$-stability}\footnote{We point out that our definition here is different from what \citet{barber2021predictive} refer to as ``in-sample stability''---see Section~\ref{sec:in_out} below for details.} as
    \[\beta^{\textnormal{in}}_{m,n}(\mathcal{A},P) = \mathbb{E}\left[\big|\widehat{\mu}_n(X_1)- \widehat{\mu}_{n+m}(X_1)\big|\right],\]
    where the fitted models $\widehat{\mu}_n,\widehat{\mu}_{n+m}$ are defined as before,
    and where the expected value is taken with respect to $(X_1,Y_1),\dots,(X_{n+m},Y_{n+m})\stackrel{\textnormal{iid}}{\sim}P$.
\end{definition}
With this alternative form of stability now defined, we are ready to state our training-conditional guarantee for the full conformal method---again with a small inflation added into the construction of the prediction set. 
\begin{theorem}[Training-conditional coverage for full conformal]\label{thm:CP_main}
Fix any distribution $P$ on $\mathcal{X}\times\mathbb{R}$, any symmetric algorithm $\mathcal{A}$, any sample size $n\geq 1$, and any inflation parameter $\gamma>0$. Let
\[\alpha_P^{\textnormal{$\gamma$-CP}}(\mathcal{D}_n) = \mathbb{P}\left\{ Y_{n+1}\not\in \widehat{C}_n^{\textnormal{$\gamma$-CP}}(X_{n+1})\mid \mathcal{D}_n\right\}\]
denote the training-conditional miscoverage rate of the $\gamma$-inflated full conformal prediction set,
\[\widehat{C}^{\textnormal{$\gamma$-CP}}_n(X_{n+1}) = \Big\{y \in\mathbb{R} : |y - \widehat{\mu}^y(X_{n+1})|  \leq \textnormal{Quantile}\big(\{|Y_i - \widehat{\mu}^y(X_i)|\}_{i\in[n]}; (1-\alpha)(1+1/n)\big) + \gamma\Big\}.\]
Then, for any $m\geq 1$ and any $\delta \in (0,1)$,
\[\mathbb{P}\left\{ \alpha_P^{\textnormal{$\gamma$-CP}}(\mathcal{D}_n) < \alpha + 3\sqrt{\frac{\log(1/\delta)}{2\min\{n,m\}}} + 2\sqrt[3]{\frac{2\beta^{\textnormal{in}}_{m-1,n+1}(\mathcal{A},P)}{\gamma}} \right\}  \geq 1 - 3\delta - \sqrt[3]{\frac{2\beta^{\textnormal{in}}_{m-1,n+1}(\mathcal{A},P)}{\gamma}},\]
where probability is taken with respect to $\mathcal{D}_n = \big\{(X_i,Y_i)\big\}_{i\in[n]}\stackrel{\textnormal{iid}}{\sim} P$, and where $\beta^{\textnormal{in}}_{m-1,n+1}(\mathcal{A},P)$ is defined as in Definition~\ref{def:stability_E_in}.
\end{theorem}
As for the jackknife+ result given above in Theorem~\ref{thm:jack_main}, here we see that the (inflated) full conformal prediction set offers a training-conditional coverage guarantee of the form $\alpha_P(\mathcal{D}_n)\leq \alpha + \mathrm{o}(1)$ with probability $1-\mathrm{o}(1)$, as long as the algorithm $\mathcal{A}$ satisfies in-sample $m$-stability for some large $m$.

We might ask whether, for full conformal, we can also prove a result for the uninflated prediction set (as in Theorem~\ref{thm:jack_main_uninflated}, for jackknife+, which assumes only a bounded density for $Y\mid X$). However, it is likely that a bounded density assumption is not sufficient for deriving an analogous result for full conformal. This is because for jackknife+, the prediction set is in fact an interval, and so inflated set $\widehat{C}_n^{\textnormal{$\gamma$-J+}}(X_{n+1})$ and its uninflated counterpart $\widehat{C}_n^{\textnormal{J+}}(X_{n+1})$ differ only by Lebegue measure $2\gamma$---that is, $\textnormal{Leb}(\widehat{C}_n^{\textnormal{$\gamma$-J+}}(X_{n+1}) \backslash \widehat{C}_n^{\textnormal{J+}}(X_{n+1})) = 2\gamma$. For full conformal, since there is no equivalent simple characterization of how the sets $\widehat{C}_n^{\textnormal{CP}}(X_{n+1})$ and $\widehat{C}_n^{\textnormal{$\gamma$-CP}}(X_{n+1})$ differ, the measure of the set difference $\textnormal{Leb}(\widehat{C}_n^{\textnormal{$\gamma$-CP}}(X_{n+1}) \backslash \widehat{C}_n^{\textnormal{CP}}(X_{n+1}))$ could potentially be arbitrarily large.

\subsection{An asymptotic view}\label{sec:asymp_theory}
With our finite-sample results in place, we now return to the open question posed by \citet{bian2023training}, which asked whether training-conditional coverage can be guaranteed once we assume algorithmic stability (implicitly, $m$-stability with $m=1$, since that is the notion of stability considered by \citet{bian2023training}).

At first glance, it might appear that our results above do not resolve this question: after all, the upper bounds in Theorems~\ref{thm:jack_main} and~\ref{thm:CP_main} are meaningful only if we choose $m \gg 1$, and not for $m=1$. However, we will now see that, from an asymptotic standpoint, our results \emph{do} imply a positive answer to \citet{bian2023training}'s question.

To begin, let us return to \citet{bian2023training}'s question, which we can phrase as follows:
\begin{quote}
    If algorithm $\mathcal{A}$ satisfies a stability assumption, with $\beta^{\textnormal{out}}_{1,n}(\mathcal{A},P) = \mathrm{o}(1)$, then does jackknife+  satisfy a training-conditional coverage guarantee at level $1-\alpha - \mathrm{o}_p(1)$?
\end{quote}
For full conformal, we can also ask the analogous question with in-sample stability in place of out-of-sample stability.

We will now see that our theoretical guarantees that rely on $m$-stability, can actually ensure a guarantee of this seemingly stronger type---that is, asymptotically, training-conditional coverage requires only $1$-stability (with a vanishing value of $\beta^{\textnormal{out}}_{1,n}(\mathcal{A},P)$ or $\beta^{\textnormal{in}}_{1,n}(\mathcal{A},P)$). To ease notation in the theorem statements, we define $\beta^{*,\textnormal{out}}_{1,n} = \sup_{n'\geq n}\beta^{\textnormal{out}}_{1,n'}$, and similarly $\beta^{*,\textnormal{in}}_{1,n} = \sup_{n'\geq n}\beta^{\textnormal{in}}_{1,n'}$. Since, in general, we expect stability to improve with increasing sample size $n$, in practical settings we would likely have $\beta^{*,\textnormal{out}}_{1,n}={\beta}^{\textnormal{out}}_{1,n}$, and same for in-sample stability.
\begin{theorem}[Asymptotic training-conditional coverage for jackknife+]\label{thm:jack_asymp}
Fix any sequences $b_n,\gamma_n>0$ such that $b_n/\gamma_n\to 0$. Then there exists a sequence $\epsilon_n\rightarrow 0$ (constructed explicitly in the proof) such that
\[\lim \limits_{n \to \infty} \sup \limits_{\substack{\mathcal{A}, P\textnormal{ s.t. }\\ \beta_{1,n-1}^{*,\textnormal{out}}(\mathcal{A},P) \le b_{n}}}\mathbb{P}\left\{\alpha_{P}^{\textnormal{$\gamma_n$-J+}}(\mathcal{D}_n)\geq \alpha + \epsilon_n\right\} =0.\]

Moreover, a similar result holds without inflation, under an additional bounded density assumption. For each $B>0$, define a set $\mathcal{P}_B$ of distributions on $(X,Y)\in\mathcal{X}\times\mathbb{R}$, given by
\[\mathcal{P}_B = \Big\{\textnormal{distributions $P$ such that $Y\mid X$ has a conditional density}  \textnormal{satisfying $\mathbb{E}\left[\sup_{y\in\mathbb{R}}f_{Y|X}(y\mid X)\right]\leq B$}\Big\}.\]
Fix any sequences $b_n,B_n>0$ such that $b_n B_n\to 0$. Then there exists a sequence $\epsilon_n\rightarrow 0$ (constructed explicitly in the proof) such that
\[\lim \limits_{n \to \infty} \sup \limits_{\substack{\mathcal{A},P \textnormal{ s.t. $P \in \mathcal{P}_{B_n}$,}\\ \beta_{1,n-1}^{*,\textnormal{out}}(\mathcal{A},P) \le b_{n}}}\mathbb{P}\left\{\alpha_{P}^{\textnormal{J+}}(\mathcal{D}_n)\geq \alpha + \epsilon_n\right\} =0.\]

\end{theorem}
\begin{theorem}[Asymptotic training-conditional coverage for full conformal]\label{thm:CP_asymp} 
Fix any sequences $b_n,\gamma_n>0$ such that $b_n/\gamma_n\to 0$. Then there exists a sequence $\epsilon_n\rightarrow 0$ (constructed explicitly in the proof) such that
\[\lim \limits_{n \to \infty} \sup \limits_{\substack{\mathcal{A}, P\textnormal{ s.t. }\\ \beta_{1,n+1}^{*,\textnormal{in}}(\mathcal{A},P) \le b_{n}}}\mathbb{P}\left\{\alpha_{P}^{\textnormal{$\gamma_n$-CP}}(\mathcal{D}_n)\geq \alpha + \epsilon_n\right\} =0.\]
\end{theorem}

To help parse these asymptotic results, consider fixing the training algorithm $\mathcal{A}$ and data distribution $P$. Then, these results show us that, as long as $1$-stability holds with some vanishing parameter $\beta_{1,n}(\mathcal{A},P)$ (i.e., out-of-sample or in-sample stability, as appropriate, for jackknife+ or for full conformal), then we can always choose an inflation value $\gamma_n$, with $\gamma_n\to0$ sufficiently slowly,
such that asymptotic training-conditional coverage will hold (or, if we additionally assume $Y\mid X$ has a continuous distribution, then for jackknife+ no inflation is necessary). Thus we can conclude that, from this asymptotic view, $1$-stability is indeed sufficient for establishing training-conditional coverage. To give a preview of the proofs, the idea is that we take a sequence $m_n\rightarrow\infty$ that grows sufficiently slowly, and relate $m_n$-stability (which is needed for applying the finite-sample theorems) to $1$-stability (which is assumed in the asymptotic theorems).

\subsection{Related literature on stability and jackknife+}\label{sec:prior_jack+}
As mentioned in Sections~\ref{sec:intro} and~\ref{sec:background}, several earlier works have also considered the connection between algorithmic stability and the coverage properties of the jackknife+ (or the jackknife). Here we will give more details on these various results, to compare to our present work.

First we consider \emph{marginal} coverage. 
Under an algorithmic stability assumption,  as mentioned in Section~\ref{sec:stability_connection_to_DF}, \citet{barber2021predictive} established that the jackknife+ offers marginal coverage at level $1-\alpha-\mathrm{o}(1)$, i.e., without the factor of 2 appearing in the assumption-free marginal coverage guarantees summarized in Section~\ref{sec:background}. \citet{barber2021predictive} also show a similar result for the jackknife, rather than jackknife+; earlier work by \citet{steinberger2023conditional} also established an asymptotic version of this result for the jackknife. Concretely, adapting the result of \citet[Theorem 5]{barber2021predictive} to our notation, the $\gamma$-inflated jackknife+ interval satisfies a marginal coverage guarantee
\[\mathbb{E}\left[\alpha_P^{\textnormal{$\gamma$-J+}}(\mathcal{D}_n)\right] \leq \alpha +\mathcal{O}\left(\sqrt{\frac{\beta^{\textnormal{out}}_{1,n-1}(\mathcal{A},P)}{\gamma}}\right),\]
and the same is true for the jackknife as well. In particular, this implies that if $\mathcal{A}$ satisfies out-of-sample $1$-stability, i.e., $\beta^{\textnormal{out}}_{1,n-1}(\mathcal{A},P)$ is small, then a small inflation is sufficient to ensure marginal coverage (without a factor of 2) for both jackknife and jackknife+. Moreover, if $Y\mid X$ has a continuous distribution, \citet{barber2021predictive} establish guarantees for the uninflated jackknife+ (or jackknife) interval.

Next, we turn to related results on \emph{training-conditional} coverage. First, as mentioned in Section~\ref{sec:intro}, \citet{steinberger2023conditional} prove that, under some additional assumptions, out-of-sample $1$-stability yields asymptotic training-conditional coverage for the jackknife. These results are strengthened by the recent work of \citet{amann2023assumption},\footnote{\citet{amann2023assumption}'s work appeared online in parallel with this paper.} which establishes finite-sample guarantees for training-conditional coverage, for both jackknife and jackknife+ (see, e.g., \citet[Proposition 5.4, Corollary 5.6]{amann2023assumption}). The results of \citet{amann2023assumption} are complementary to our theorems, as they study the problem from a different framework, although both papers establish finite-sample and asymptotic results for jackknife+ under the same type of stability assumption. Their work allows for nonsymmetric regression algorithms $\mathcal{A}$, and uses some additional mild distributional assumptions to enable concentration type properties of the leave-one-out residuals, leading to their training-conditional coverage guarantees. In contrast, our work uses a new exchangeability-based argument, which means that we do require symmetry of $\mathcal{A}$, but on the other hand are able to extend the technique to prove results for full conformal as well.

\section{Alternative frameworks for stability}\label{sec:stability_frameworks}
In this section, we will examine several different options for defining stability, to understand the different frameworks appearing in the literature and to see how these frameworks relate to the results obtained in our work.

\subsection{Stability in expectation versus in probability}\label{sec:E_or_P}
Our stability definitions explore the expected change in the prediction, e.g.,  out-of-sample stability requires a bound on
\begin{equation}\label{eqn:E_version}\mathbb{E}\left[\left|\widehat\mu_n(X) - \widehat\mu_{n+m}(X)\right|\right]\leq\beta_{m,n}.\end{equation}
For the case $m=1$, this type of definition appears in \citet{bousquet2002stability, elisseeff2005stability}. 
In contrast, we could instead consider a probabilistic bound on the same quantity,
\begin{equation}\label{eqn:P_version}\mathbb{P}\left\{\left|\widehat\mu_n(X) - \widehat\mu_{n+m}(X)\right|\leq \epsilon_{m,n}\right\}\geq 1-\nu_{m,n}.\end{equation}
For the case $m=1$, this type of bound appears in many works in the literature, including \citet{kearns1997algorithmic, soloff2023bagging}.

In fact, the two can be viewed as essentially equivalent (at least, for the case of bounded response $Y$): the probabilistic bound~\eqref{eqn:P_version} implies that $\beta_{m,n}\leq \epsilon_{m,n}+\nu_{m,n}$ if (say) predictions lie in $[0,1]$ almost surely; conversely, by Markov's inequality, the probabilistic bound~\eqref{eqn:P_version} holds with any $\epsilon_{m,n}\nu_{m,n} = \beta_{m,n}$. 

Under the probabilistic version of $m$-stability, as in~\eqref{eqn:P_version}, our theorems can be modified to still hold even without assuming that the response $Y$ is bounded. For completeness, we include these results here:
\begin{theorem}[Training-conditional coverage for jackknife+---alternative version]\label{thm:jack_version_P}
Under the setting and notation of Theorem~\ref{thm:jack_main}, assume also that $\mathcal{A}$ satisfies
\[\mathbb{P}\left\{\left|\widehat\mu_{n-1}(X) - \widehat\mu_{n+m-1}(X)\right|\leq \epsilon_{m,n-1}\right\} \geq 1- \nu_{m,n-1}.\]
Then
\[\mathbb{P}\left\{ \alpha_P^{\textnormal{$2\epsilon_{m,n-1}$-J+}}(\mathcal{D}_n) < \alpha + 3\sqrt{\frac{\log(1/\delta)}{2\min\{n,m\}}} + 2\sqrt[3]{2\nu_{m,n-1}} \right\} \geq 1 - 3\delta - \sqrt[3]{2\nu_{m,n-1}} .\]
Moreover, if the conditional distribution of $Y\mid X$ has density $f_{Y|X}$, then defining $B_{\textnormal{dens}}$ as in Theorem~\ref{thm:jack_main_uninflated} we have
\[\mathbb{P}\left\{ \alpha_P^{\textnormal{J+}}(\mathcal{D}_n) < \alpha + 3\sqrt{\frac{\log(1/\delta)}{2\min\{n,m\}}} + 2\sqrt[3]{2\nu_{m,n-1}} + 4B_{\textnormal{dens}}\cdot \epsilon_{m,n-1}\right\} \geq 1 - 3\delta - \sqrt[3]{2\nu_{m,n-1}} .\]
\end{theorem}
\begin{theorem}[Training-conditional coverage for full conformal---alternative version]\label{thm:CP_version_P}
Under the setting and notation of Theorem~\ref{thm:CP_main}, assume also that $\mathcal{A}$ satisfies
\[\mathbb{P}\left\{\left|\widehat\mu_{n+1}(X_1) - \widehat\mu_{n+m}(X_1)\right|\leq \epsilon_{m-1,n+1}\right\} \geq 1- \nu_{m-1,n+1}.\]
Then
\[\mathbb{P}\left\{ \alpha_P^{\textnormal{$2\epsilon_{m-1,n+1}$-CP}}(\mathcal{D}_n) < \alpha + 3\sqrt{\frac{\log(1/\delta)}{2\min\{n,m\}}} + 2\sqrt[3]{2\nu_{m-1,n+1}} \right\} \geq 1 - 3\delta - \sqrt[3]{2\nu_{m-1,n+1}} .\]
\end{theorem}
\subsection{Stability when data points are swapped rather than removed/added}\label{sec:swap_preview}
The definitions of stability considered in this paper (and in much of the literature) examine the changes in predictions when we add (or equivalently, remove) a small number of training points to the fitted model, i.e., comparing models fitted on training sets of size $n$ and $n+m$. A similar notion of stability asks whether the trained model changes substantially if we \emph{replace}, rather than add or remove, $m$ many training points. Specifically, for $n\geq m \geq 0$, define
\[\bar{\beta}_{m,n}^{\textnormal{out}}(\mathcal{A},P) = \mathbb{E}\left[\left|\widehat\mu_n(X) - \widehat\mu_n'(X)\right|\right]\]
and (for $m<n$)
\[\bar{\beta}_{m,n}^{\textnormal{in}}(\mathcal{A},P) = \mathbb{E}\left[\left|\widehat\mu_n(X_1) - \widehat\mu_n'(X_1)\right|\right]\]
where 
\[\widehat\mu_n = \mathcal{A}\left((X_1,Y_1),\dots,(X_n,Y_n)\right), \] \[\widehat\mu_n' = \mathcal{A}\left((X_1,Y_1),\dots,(X_{n-m},Y_{n-m}),(X'_{n-m+1},Y'_{n-m+1}),\dots,(X'_n,Y'_n)\right),\]
and where the expected values are taken with respect to all data points $(X,Y),(X_i,Y_i),(X'_i,Y'_i)$ sampled i.i.d.\ from $P$.

Our previous definition of (in- or out-of-sample) stability immediately implies this swapped version of stability as well:
\begin{proposition}\label{prop:stable_implies_swap_stable}
For any $\mathcal{A}$ and any $P$, it holds that
\[\bar{\beta}_{m,n}^{\textnormal{out}}(\mathcal{A},P) \leq 2\beta_{m,n}^{\textnormal{out}}(\mathcal{A},P),\]
and
\[\bar{\beta}_{m,n}^{\textnormal{in}}(\mathcal{A},P) \leq 2\beta_{m,n}^{\textnormal{in}}(\mathcal{A},P).\]
\end{proposition}
\begin{proof}[Proof of Proposition~\ref{prop:stable_implies_swap_stable}]
    Let \[\widehat\mu_{n+m} = \mathcal{A}((X_1,Y_1),\dots,(X_n,Y_n),(X'_{n-m+1},Y'_{n-m+1}),\dots,(X'_n,Y'_n)).\]
    For out-of-sample stability, we have
    \[\mathbb{E}\left[\left|\widehat\mu_n(X) - \widehat\mu_n'(X)\right|\right] \leq \mathbb{E}\left[\left|\widehat\mu_n(X) - \widehat\mu_{n+m}(X)\right|\right] + \mathbb{E}\left[\left|\widehat\mu_n'(X) - \widehat\mu_{n+m}(X)\right|\right] \leq 2\beta_{m,n}^{\textnormal{out}}(\mathcal{A},P),\]
    by definition of out-of-sample $m$-stability together with the assumption of symmetry of $\mathcal{A}$.
    A similar argument holds for in-sample stability as well.
\end{proof}

Of course, the converse cannot hold---as an extreme example, we can take an algorithm $\mathcal{A}$ which, when run on a training sample of size $n$, returns the constant function $\widehat\mu(x)\equiv n$. Then $m$-stability is poor (specifically, $\beta_{m,n}^{\textnormal{out}}(\mathcal{A},P) = \beta_{m,n}^{\textnormal{in}}(\mathcal{A},P)=m$), but for swap-stability, we have $\bar{\beta}_{m,n}^{\textnormal{out}}(\mathcal{A},P) = \bar{\beta}_{m,n}^{\textnormal{in}}(\mathcal{A},P)=0$. In other words, the converse fails because the algorithm $\mathcal{A}$ may behave very differently on a training sample of size $n$ versus on a training sample of size $n+m$.

Interestingly, however, we will now see that our training-conditional coverage guarantees, from Theorems~\ref{thm:jack_main} and~\ref{thm:CP_main}, nonetheless hold even with the swap versions of stability.
\begin{theorem}[Training-conditional coverage for jackknife+---swap-stability version]\label{thm:jack_swap}
The finite-sample training-conditional coverage guarantees given in Theorems~\ref{thm:jack_main} and~\ref{thm:jack_main_uninflated} for the jackknife+, hold with $\bar{\beta}_{m,n-1}^{\textnormal{out}}(\mathcal{A},P)$ in place
of $\beta_{m,n-1}^{\textnormal{out}}(\mathcal{A},P)$. Similarly, the asymptotic guarantees given in Theorem~\ref{thm:jack_asymp} hold as well, 
with $\bar\beta^{*,\textnormal{out}}_{1,n-1} = \sup_{n'\geq n-1}\bar\beta^{\textnormal{out}}_{1,n'}$ in place of $\beta^{*,\textnormal{out}}_{1,n-1}$.
\end{theorem}
\begin{theorem}[Training-conditional coverage for full conformal---swap-stability version]\label{thm:CP_swap}
The finite-sample training-conditional coverage guarantee given in Theorem~\ref{thm:CP_main} for full conformal, holds with $\bar{\beta}_{m-1,n+1}^{\textnormal{in}}(\mathcal{A},P)$ in place
of $\beta_{m-1,n+1}^{\textnormal{in}}(\mathcal{A},P)$.  Similarly, the asymptotic guarantee given in Theorem~\ref{thm:CP_asymp} holds as well, with $\bar\beta^{*,\textnormal{in}}_{1,n+1} = \sup_{n'\geq n+1}\bar\beta^{\textnormal{in}}_{1,n'}$ in place of $\beta^{*,\textnormal{in}}_{1,n+1}$.
\end{theorem}

For jackknife+, this swap-stability version of our main result holds because we can establish a weaker sort of converse to Proposition~\ref{prop:stable_implies_swap_stable}, for out-of-sample stability. For full conformal, the proof will require a different approach. See the Appendix~\ref{appB} for details.

\subsection{In-sample, out-of-sample, or neither}\label{sec:in_out}
In the results above, we have considered two different notions of stability, with respect to how the test point relates to the training data used for fitting the models. Taking $m=1$ for simplicity, we have considered:
\begin{itemize}
    \item Out-of-sample stability (for jackknife+), which compares two models $\widehat\mu_n$ and $\widehat\mu_{n+1}$ in terms of their predictions made on a test point $X$ that is not contained in the training set for either model;
    \item In-sample stability (for full conformal), which compares two models $\widehat\mu_n$ and $\widehat\mu_{n+1}$ in terms of their predictions made on a test point $X_1$ that is contained in the training set for both models.
\end{itemize}
However, there is a third option: the test point might be in the training set for one model, but not the other.
In particular, \citet{barber2021predictive} consider this alternative type of stability, which (adjusting to the notation of our present paper) can be expressed by the quantity
\[\beta^{\textnormal{in/out}}_{1,n}(\mathcal{A},P) = \mathbb{E}\left[\big|\widehat{\mu}_n(X_{n+1})- \widehat{\mu}_{n+1}(X_{n+1})\big|\right].\]
We write ``in/out'' in this definition to emphasize the fact that, in this case, the test point $X_{n+1}$ is \emph{in} the training sample for one of the models ($\widehat\mu_{n+1}$) but \emph{out} of the training sample for the other ($\widehat\mu_n$).
We remark that \citet{barber2021predictive} refer to this as ``in-sample stability'' but we change the terminology here to allow for all the various types of stability under consideration.

If the algorithm exhibits overfitting (e.g., linear regression with a very large number of features), it is likely that $\beta^{\textnormal{in/out}}_{1,n}(\mathcal{A},P)$ is large, since $\widehat\mu_{n+1}$ overfits to data point $(X_{n+1},Y_{n+1})$ while $\widehat\mu_n$ does not. This is often the case even if out-of-sample stability and in-sample stability both hold (i.e., even if $\beta^{\textnormal{out}}_{1,n}(\mathcal{A},P)$ and $\beta^{\textnormal{in}}_{1,n}(\mathcal{A},P)$ are both small). 
Indeed, the fact that $\beta^{\textnormal{in/out}}_{1,n}(\mathcal{A},P)$ is typically large for most regression algorithms is the reason why we need conformal or cross-validation type methods at all: if this measure of stability were small, then training error  would provide a good approximation to future test errors, i.e., the training residuals $|Y_i -\widehat\mu_n(X_i)|$ could be used to determine an appropriate margin of error for constructing a prediction interval of the form $\widehat{C}_n(X_{n+1}) = \widehat\mu_n(X_{n+1})\pm \textnormal{(some margin of error)}$.

\section{A closer look at $m$-stability}\label{sec:1-vs-m}
 The literature on algorithmic stability has often defined stability with respect to perturbing one training point (either adding or removing one point, or swapping one point). In contrast, in this work we consider the more general definition of $m$-stability, with $m\geq 1$ data points perturbed; setting $m=1$ recovers the more common definitions of the literature. 

In this section, we will see how $m$-stability, the key ingredient for our main results, relates to the more commonly studied $1$-stability property. Our findings will also lead to the proofs of the asymptotic training-conditional results of Theorems~\ref{thm:jack_asymp} and~\ref{thm:CP_asymp}.

\subsection{Relating $1$-stability and $m$-stability}
At a high level, it is intuitive that achieving $m$-stability, for some $m>1$, is more challenging than achieving $1$-stability.
For swap-stability, this intuition is correct, as shown in the following proposition.

\begin{proposition}\label{prop:beta_swap_order}
    Fix any algorithm $\mathcal{A}$ and distribution $P$. For any $n \geq m' > m> 0$,
    \[\bar\beta^{\textnormal{out}}_{m,n}(\mathcal{A},P) \leq \bar\beta^{\textnormal{out}}_{m',n}(\mathcal{A},P).\]
    Similarly, for any $n> m' > m> 0$,
    \[\bar\beta^{\textnormal{in}}_{m,n}(\mathcal{A},P) \leq \bar\beta^{\textnormal{in}}_{m',n}(\mathcal{A},P).\]    
\end{proposition}
On the other hand, if we define stability by removing/adding data points rather than by swapping, this may not necessarily hold, due to degenerate counterexamples  (e.g., we can take an algorithm $\mathcal{A}$ which, when run on a training sample of size $n$, returns the constant function $\widehat\mu(x)\equiv C\cdot \textnormal{mod}(n,2)$ for an arbitrary constant $C>0$; then $\beta_{1,n}^{\textnormal{out}}(\mathcal{A},P) = \beta_{1,n}^{\textnormal{in}}(\mathcal{A},P) = C$, but 
$\beta_{m,n}^{\textnormal{out}}(\mathcal{A},P) = \beta_{m,n}^{\textnormal{in}}(\mathcal{A},P) = 0$ for any even $m$). Nonetheless, in practice, we generally expect  
$\beta_{m,n}^{\textnormal{out}}(\mathcal{A},P)$ and $\beta_{m,n}^{\textnormal{in}}(\mathcal{A},P)$ to increase monotonically in $m$.

It also holds that $1$-stability can be used to verify $m$-stability, in a weak sense---at a high level, if $\mathcal{A}$ satisfies $1$-stability at some level $\beta$, then it also satisfies $m$-stability at level $m\beta$. Formally, the bound is given by the following lemma:
\begin{lemma}\label{lem:m_vs_1}
For any algorithm $\mathcal{A}$ and any distribution $P$, the following inequality holds for out-of-sample stability for any $n,m\geq 0$:
\[\beta_{m,n}^{\textnormal{out}}(\mathcal{A},P) \leq \sum_{k=n}^{n+m-1} \beta_{1,k}^{\textnormal{out}}(\mathcal{A},P).\]
Similarly, for any $n\geq m\geq 0$, 
\[\bar{\beta}_{m,n}^{\textnormal{out}}(\mathcal{A},P) \leq m \bar{\beta}_{1,n}^{\textnormal{out}}(\mathcal{A},P).\]
Moreover, if $n>m$ the same results hold with in-sample stability (or in-sample swap-stability) in place of out-of-sample stability (or out-of-sample swap-stability).
\end{lemma}

In particular, this lemma immediately implies the results of Theorems~\ref{thm:jack_asymp} and~\ref{thm:CP_asymp}. (We defer the detailed proof to the Appendix~\ref{appA_asy_proof}.) However, as we will demonstrate below both empirically (in Section~\ref{sec:1-vs-m-sim}) and theoretically (in Section~\ref{sec:1-vs-m-example}), this result is fundamentally somewhat weak: it does not capture the favorable stability properties that appear to hold in practice, and we will see that the $m$-stability of algorithms typically tends to be much better than this coarse upper bound.

\subsection{An empirical comparison}\label{sec:1-vs-m-sim}
As mentioned above, the upper bound given in Lemma~\ref{lem:m_vs_1}, which effectively tells us that $1$-stability at level $\beta$ implies $m$-stability at level $m\beta$, is loose---it is often the case that the actual out-of-sample or in-sample $m$-stability of an algorithm is far better than this coarse upper bound. To see this, we run a simulation to examine the out-of-sample stability for four algorithms: $k$-nearest neighbors, ridge regression, decision trees, and random forest. Figure~\ref{fig:1-vs-m-sim} shows the simulation results. For $k$-nearest neighbors, the upper bound is actually fairly tight: we see only a mild gap between the actual value of $\beta_{m,n}^{\textnormal{out}}$ and the upper bound provided by the lemma. For the other three algorithms, however, the gap is substantial, showing that the upper bound of the lemma can be quite pessimistic and in particular, showing that the algorithm $\mathcal{A}$ may indeed lead to strong training-conditional coverage results, according to our theorems.

\begin{figure}
\centering
    \includegraphics[width=\textwidth]{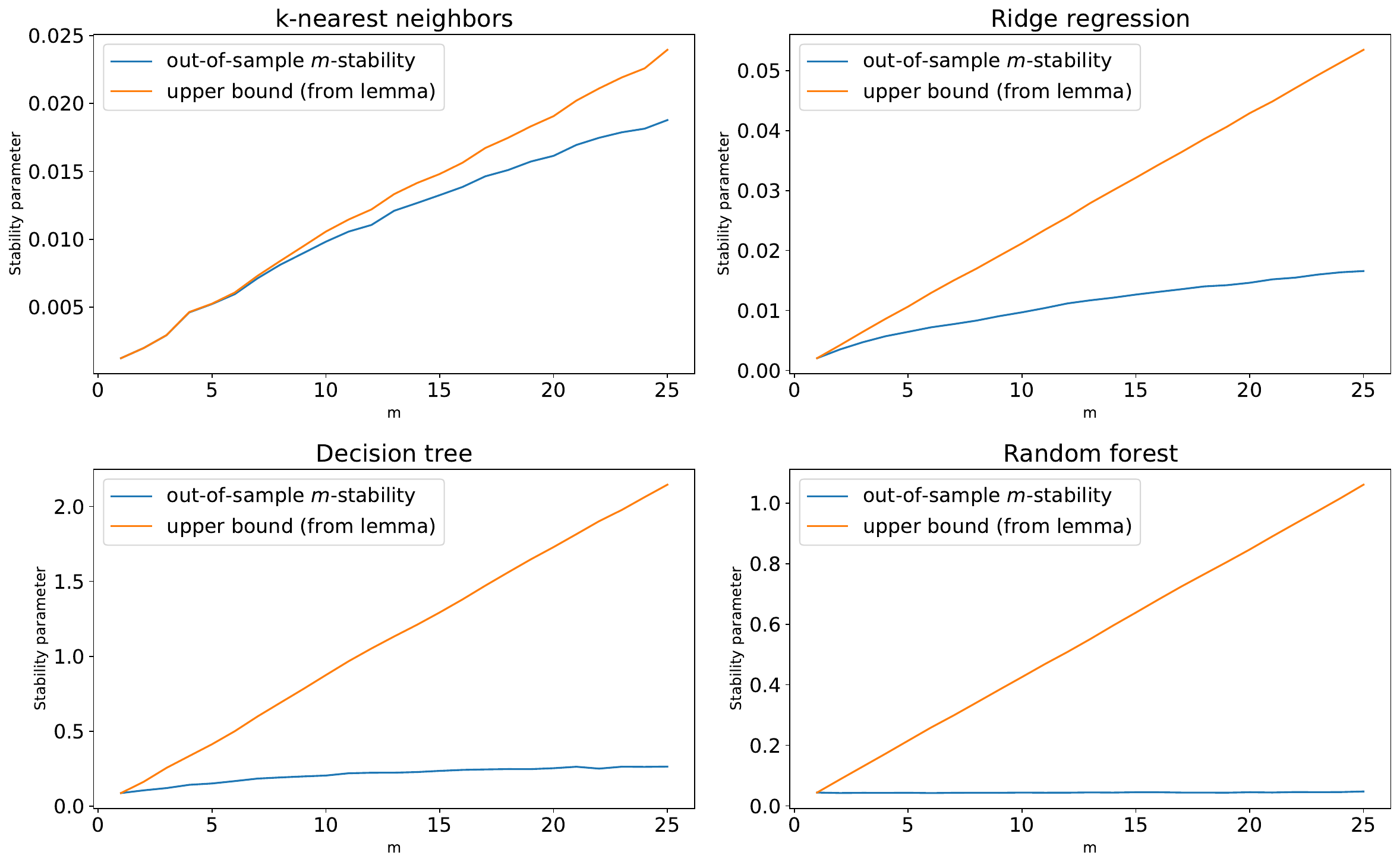}
    \caption{Comparison between the out-of-sample $m$-stability, $\beta_{m,n}^{\textnormal{out}}(\mathcal{A},P)$, and the upper bound on the $m$-stability guaranteed by Lemma~\ref{lem:m_vs_1}, for several choices of the algorithm $\mathcal{A}$. We set $n=500$ for all simulations, and the stability parameters are estimated by averaging over 1000 independent trials. See Section~\ref{sec:1-vs-m-sim} for details.}\label{fig:1-vs-m-sim}
\end{figure}

The details of the simulations are as follows.\footnote{Code to reproduce this simulation is available at \url{https://colab.research.google.com/drive/1XZiB7M9CvZIqqJRVP0dK_l8-aVEkUW-s?usp=sharing}} In each case, we generate data points by sampling i.i.d.\ from the distribution $P$ on $(X,Y)$ given by the following distribution:
\[
X \sim \mathrm{Unif}([0,1]^{d}),\
Y =  \sum_{j=1}^{d} \sin \left(\frac{X_j}{j} \right) + \epsilon\]
where
\[\epsilon\sim \begin{cases} \mathrm{Unif}[-1,1], & \text{ with probability $1/3$},\\
\mathrm{Unif}[-0.1,0.1], & \text{ with probability $2/3$.}\end{cases}
\]
Now we specify the algorithm $\mathcal{A}$ for each of the four cases. Given a training set $(X_1,Y_1),\dots,$ $(X_n,Y_n)$ (for any sample size $n$), the fitted model $\widehat\mu = \mathcal{A}((X_1,Y_1),\dots,(X_n,Y_n))$ is given by:
\begin{itemize}
    \item For $k$-nearest neighbors,
    \begin{equation}\label{eqn:define_kNN}\widehat\mu(x) = \frac{1}{k}\sum_{i\in N_{k;n}(x)}Y_i,\end{equation}
    where $N_{k;n}(x)\subseteq[n]$ contains the indices of the $k=20$ nearest neighbors of $x$, i.e., the $k$ values of $i$ that achieve the smallest distances $\|x - X_i\|_2$ (with ties broken arbitrarily).
    \item For ridge regression,
    \begin{equation}\label{eqn:define_ridge}\widehat\mu(x) = x^\top\widehat\theta,\textnormal{ where }
    \widehat\theta = \arg\min_{\theta\in\mathbb{R}^p}\left\{\frac{1}{n}\sum_{i\in[n]}(Y_i - X_i^\top\theta)^2 + \lambda\|\theta\|^2_2\right\},\end{equation}
    for penalty parameter $\lambda = 0.01$.
    \item For the decision tree, we use the \texttt{DecisionTreeRegressor} function in Python's \texttt{sklearn} package, with parameter \texttt{max\_depth} set to 8.
    \item For random forests, we use the \texttt{RandomForestRegressor} function in Python's \texttt{sklearn} package, with parameter \texttt{max\_depth} set to 8.
\end{itemize}
For each algorithm, we take $n=500$ and $d=40$, and run $1000$ independent trials to empirically estimate the out-of-sample $m$-stability $\beta_{m,n}^{\textnormal{out}}(\mathcal{A},P)$, and the upper bound $\sum_{k=n}^{n+m-1}\beta_{1,k}^{\textnormal{out}}(\mathcal{A},P)$ given in Lemma~\ref{lem:m_vs_1}, for each $m=1,\dots,25$.

\subsection{Examples of $m$-stable algorithms: theoretical results}\label{sec:1-vs-m-example}
In this section, we will examine three types of regression algorithms, to establish theoretical bounds on their $m$-stability. In particular, we will see that, for some of them, the scaling as $m$ increases behaves more favorably than the coarse bound given above in Lemma~\ref{lem:m_vs_1}.
This parallels some of the findings in our simulation in Section~\ref{sec:1-vs-m-sim} above, which showed several practical algorithms that exhibit good $m$-stability, beyond the coarse bound predicted by Lemma~\ref{lem:m_vs_1}.

\subsubsection{$k$-nearest neighbors}
First, we consider the $k$-nearest neighbors algorithm, as defined in~\eqref{eqn:define_kNN} above.
The following result computes a bound on the $m$-stability of this algorithm.
    \begin{proposition}\label{prop:kNN}
        For any distribution $P$ on $\mathcal{X}\times\mathbb{R}$ (where $\mathcal{X}$ is a metric space), any $m\geq 1$, and any $1\leq k\leq n$, if $|Y|\leq B_Y$ holds almost surely under $P$ then the $k$-nearest neighbors algorithm satisfies \[\max\{\beta_{m,n}^{\textnormal{out}}(\mathcal{A},P), \beta_{m,n}^{\textnormal{in}}(\mathcal{A},P) \} \leq 2B_Y \cdot \frac{m}{n+m}.\]
    \end{proposition} 
In particular, this result suggests that for $m\leq n$, we have $\beta_{m,n}(\mathcal{A},P) = \mathcal{O}(\frac{m}{n})$ for both in- and out-of-sample stability, which means that $\beta_{m,n}(\mathcal{A},P)  \propto m \cdot \beta_{1,n}(\mathcal{A},P)$, as suggested by the upper bound proved in Lemma~\ref{lem:m_vs_1}. This agrees with the findings of our simulation in Section~\ref{sec:1-vs-m-sim}, where we observed that the bound given in Lemma~\ref{lem:m_vs_1} is relatively tight for the $k$-nearest neighbors algorithm (unlike the other algorithms).

\subsubsection{Ridge regression}
Next, we consider ridge regression, as defined in~\eqref{eqn:define_ridge}. 
\begin{proposition}\label{prop:ridge}
        For any distribution $P$ on $\mathbb{R}^d\times\mathbb{R}$, any $m\geq 1$, and any $\lambda>0$, if $\|X\|_2\leq B_X$ and $|Y|\leq B_Y$ hold almost surely under $P$, then the ridge regression algorithm satisfies \[\max\{\beta_{m,n}^{\textnormal{out}}(\mathcal{A},P), \beta_{m,n}^{\textnormal{in}}(\mathcal{A},P) \}\leq \frac{2B_X^2B_Y}{\lambda}\left(1 + \frac{B_X^2}{\lambda}\right) \cdot \frac{m}{n+m} \cdot\left(\frac{1}{\sqrt{m}}+\frac{1}{\sqrt{n}}\right).\]
\end{proposition}
Observe that, for $m\leq n$ (and treating $\lambda$ as a constant), this result suggests that $\beta_{m,n}(\mathcal{A},P) = \mathcal{O}(\frac{\sqrt{m}}{n})$ (again, for both in- and out-of-sample stability). In particular, this means that $\beta_{m,n}(\mathcal{A},P)\ll m \cdot \beta_{1,n}(\mathcal{A},P)$, meaning that the upper bound of Lemma~\ref{lem:m_vs_1} is likely to be extremely loose for large $m$. This is corroborated by the simulation in Section~\ref{sec:1-vs-m-sim}, where for ridge regression, the bound given in Lemma~\ref{lem:m_vs_1} is relatively loose empirically.

\subsubsection{Bagging}
Finally, we will consider bagged algorithms. Bagging refers to an approach towards stabilizing an algorithm $\mathcal{A}$ by re-running it multiple times on perturbed or subsampled versions of the data set, then averaging the output; see, e.g., \citet{breiman1996bagging,dietterich2000ensemble,breiman2001random} for background on this well-known technique. Concretely, given a base algorithm $\mathcal{A}$, we will consider the following definition of its bagged version $\mathcal{A}_{\textnormal{bag}}$.\footnote{Here we work with a specific implementation of bagging, sometimes referred to as ``subbagging'' in the literature, where the resampled data sets are drawn via sampling without replacement. Similar results also hold for other resampling schemes, such as sampling with replacement (often called ``bootstrapping'').} For a data set $\mathcal{D} = \{(X_i,Y_i)\}_{i\in[n]}$, $\mathcal{A}_{\textnormal{bag}}$ is a (randomized) procedure that samples $B$ many data sets of size $N$,
\[\mathcal{D}^{(b)} = \{(X_{i^{(b)}_j},Y_{i^{(b)}_j})\}_{j\in[N]},\]
where $i^{(b)}_1,\dots,i^{(b)}_N\in[n]$ is a subsample drawn uniformly at random from $[n]$ without replacement, e.g., for $N=n/2$. Writing $\widehat{\mu}^{(b)} = \mathcal{A}(\mathcal{D}^{(b)})$ as the fitted model for each subsampled data set, the output of the bagged algorithm is then given by the function that returns the averaged prediction, \[\widehat{\mu}(x) = \frac{1}{B}\sum_{b=1}^B \widehat\mu^{(b)}(x).\]

 Recent work by \citet{soloff2023bagging,soloff2024stability} establishes the $1$-stability of bagging for \emph{any} base algorithm $\mathcal{A}$---in particular, if we assume that models produced by $\mathcal{A}$ return outputs lying in $[0,1]$, \citet[Corollary 15]{soloff2023bagging} implies that
 \[\beta^{\textnormal{out}}_{1,n}(\mathcal{A}_{\textnormal{bag}},P)\leq \sqrt{\frac{1}{4n}\cdot \frac{N}{n+1-N}} + \mathcal{O}\left(\frac{1}{\sqrt{B}}\right).\]
 The following proposition extends this result to $m$-stability for a general $m$. For simplicity, we only consider out-of-sample stability in this result; similar techniques can also establish slightly weaker bounds for $\beta_{m,n}^{\textnormal{in}}$.
 \begin{proposition}\label{prop:bagging}
    For any distribution $P$ on $\mathcal{X}\times\mathbb{R}$, any $m\geq 1$, any $B\geq 1$, and any $1\leq N<n$, if the base algorithm $\mathcal{A}$ returns outputs lying in $[0,1]$, then its bagged version $\mathcal{A}_{\textnormal{bag}}$ satisfies
    \[\beta_{m,n}^{\textnormal{out}}(\mathcal{A}_{\textnormal{bag}},P)\leq  \sqrt{\frac{m}{4n}\cdot \frac{N}{n+1-N}} + \frac{1}{\sqrt{B}}.\]
 \end{proposition}
If we take $\frac{N}{n-N}=\mathcal{O}(1)$ as is typically the case in practice (e.g., by choosing $N=n/2$), we therefore see that $\beta_{m,n}^{\textnormal{out}}(\mathcal{A}_{\textnormal{bag}},P) = \mathcal{O}(\sqrt{\frac{m}{n}} + \frac{1}{\sqrt{B}})$, meaning that we may have $\beta_{m,n}^{\textnormal{out}}(\mathcal{A}_{\textnormal{bag}},P) \ll m \cdot \beta_{m,n}^{\textnormal{out}}(\mathcal{A}_{\textnormal{bag}},P)$, particularly if $B$ is small---in other words, the upper bound of Lemma~\ref{lem:m_vs_1} may be loose. This is supported by our simulation results in Section~\ref{sec:1-vs-m-sim} for the random forest algorithm, which is a specific instance of bagging---namely, by taking the base algorithm $\mathcal{A}$ to be a regression tree---although in that case, resampled data sets are constructed by sampling with replacement rather than without replacement.

\section{Proofs of main results}\label{sec:proofs}
We now turn to the proofs of our main results, Theorems~\ref{thm:jack_main},~\ref{thm:jack_main_uninflated}, and~\ref{thm:CP_main}, which establish finite-sample training-conditional coverage guarantees for the jackknife+  (with and without inflation) and for full conformal, respectively.  All other proofs are in the Appendix.

\subsection{Proof of Theorem~\ref{thm:jack_main} (jackknife+)}
First, consider a data set $(X_1,Y_1),\dots,(X_{n+m},Y_{n+m})$ consisting of $n+m$ data points drawn i.i.d.\ from $P$. The idea is that data points $i=1,\dots,n$ are the training data, as before, while data points $i=n+1,\dots,n+m$ act as test points.

\subsubsection{Step 1: training-conditional coverage versus coverage on the test set}
First we will relate the training-conditional miscoverage  $\alpha_P^{\textnormal{$\gamma$-J+}}(\mathcal{D}_n)$ to miscoverage for the test set. Let
\begin{multline*}\widehat{C}^{\textnormal{$\gamma$-J+}}_n(x) = \Big[-\textnormal{Quantile}\big(\{-\widehat{\mu}_{-i}(x) + R_i \}_{i\in[n]};(1-\alpha)(1+1/n)\big) -\gamma, \\\textnormal{Quantile}\big(\{\widehat{\mu}_{-i}(x) + R_i \}_{i\in[n]};(1-\alpha)(1+1/n)\big)+ \gamma\Big]
\end{multline*}
be the $\gamma$-inflated jackknife+ prediction interval, trained on $\{(X_j,Y_j)\}_{j\in[n]}$, at a test point $x$. Then $\alpha_P^{\textnormal{$\gamma$-J+}}(\mathcal{D}_n)$ is the miscoverage rate for this interval, i.e.,
\[ \alpha_P^{\textnormal{$\gamma$-J+}}(\mathcal{D}_n) = \mathbb{P}\{Y\not\in \widehat{C}^{\textnormal{$\gamma$-J+}}_n(X)\mid \mathcal{D}_n\},\]
with probability taken for a new data point $(X,Y)\sim P$. Since the data points are i.i.d., then conditional on $\mathcal{D}_n$, the random variable
\[S_{\textnormal{test}} = \sum_{i=n+1}^{n+m}\mathbf{1}\{Y_i\not\in \widehat{C}^{\textnormal{$\gamma$-J+}}_n(X_i)\} \]
is Binomial, with $m$ draws, each with probability $\alpha_P^{\textnormal{$\gamma$-J+}}(\mathcal{D}_n)$. 
Therefore by Hoeffding's inequality, for any $\delta \in (0,1)$,
\[\mathbb{P}\left\{S_{\textnormal{test}} \leq m\left(\alpha_P^{\textnormal{$\gamma$-J+}}(\mathcal{D}_n) - \sqrt{\frac{\log(1/\delta)}{2m}}\right) \ \middle| \  \mathcal{D}_{n} \right\} \leq \delta.\]
Marginalizing over $\mathcal{D}_{n}$, we can then equivalently write
\[\mathbb{P}\left\{\alpha_P^{\textnormal{$\gamma$-J+}}(\mathcal{D}_n) < \frac{S_{\textnormal{test}}}{m} + \sqrt{\frac{\log(1/\delta)}{2m}}\right\} \geq 1-\delta.\]
From this point on, it suffices to verify that $S_{\textnormal{test}}$ is not much larger than $\alpha m$, with high probability.

\subsubsection{Step 2: define an oracle}
The jackknife+ interval is built using the leave-one-out models,
\[\widehat\mu_{-i} = \mathcal{A}\left(\{(X_j,Y_j)\}_{j\in[n]\backslash\{i\}}\right).\] 
For this part of the proof, we will define a new collection of models, fitted via leave-one-out regression on a larger data set:
\[\widetilde{\mu}_{-i} = \mathcal{A}\left(\{(X_j,Y_j)\}_{j\in[n+m]\backslash\{i\}}\right).\]
(The $m$-stability assumption on $\mathcal{A}$ ensures that we should have $\widehat\mu_{-i} \approx \widetilde{\mu}_{-i}$, for each $i\in[n]$; we will use this property shortly.)
Define the leave-one-out residuals for these larger models,
\[\widetilde{R}_i = |Y_i - \widetilde{\mu}_{-i}(X_i)|, \ i\in[n+m].\]

Fix $\alpha'\in[0,1]$ (to be defined below) and let
\[Q = \textnormal{Quantile}\left(\{\widetilde{R}_i\}_{i\in[n+m]};1-\alpha'\right).\]
Define
\[A = \sum_{i\in[n+m]}\mathbf{1}\left\{\widetilde{R}_i > Q\right\}, \quad B = \sum_{i\in[n+m]}\mathbf{1}\left\{\widetilde{R}_i \geq Q\right\}.\]
Then it must hold surely that
$A\leq \alpha'(n+m) \leq B$,
by definition of the quantile. Define also
\[A_{\textnormal{test}} = \sum_{i=n+1}^{n+m}\mathbf{1}\left\{\widetilde{R}_i > Q\right\}, \quad B_{\textnormal{train}} = \sum_{j\in[n]}\mathbf{1}\left\{\widetilde{R}_j \geq Q\right\}.\]
By exchangeability of the data and symmetry of the algorithm $\mathcal{A}$, the residuals $\widetilde{R}_1,\dots,\widetilde{R}_{n+m}$ are exchangeable. In other words, the test residuals $\widetilde{R}_{n+1},\dots,\widetilde{R}_{n+m}$ can be viewed as a sample drawn uniformly without replacement from the full list of $n+m$ residuals, and same for the training residuals. More formally, we have
$A_{\textnormal{test}}\mid A\sim\textnormal{HyperGeometric}(n+m,m,A)$ and $B_{\textnormal{train}}\mid B\sim\textnormal{HyperGeometric}(n+m,n,B)$.
By Hoeffding's inequality for sampling without replacement, then, for any $\delta \in (0,1)$,
\[\mathbb{P}\left\{A_{\textnormal{test}} < A\cdot \frac{m}{n+m}+ \sqrt{\frac{m}{2}\log(1/\delta)}\right\} \geq 1- \delta\]
and
\[\mathbb{P}\left\{B_{\textnormal{train}} > B\cdot \frac{n}{n+m} - \sqrt{\frac{n}{2}\log(1/\delta)}\right\} \geq 1- \delta.\]
Now suppose that both of these events hold. Then first, by our bound on $B_{\textnormal{train}}$, we have
\[\sum_{j\in[n]} \mathbf{1}\left\{\widetilde{R}_j \geq Q\right\} = B_{\textnormal{train}} > B \cdot\frac{n}{n+m} - \sqrt{\frac{n}{2}\log(1/\delta)} \geq \alpha' n  - \sqrt{\frac{n}{2}\log(1/\delta)} = \alpha''n,\]
where $\alpha'' = \alpha' - \sqrt{\log(1/\delta)/2n}$. In particular, this implies that, on the above events,
\[\textnormal{Quantile}\left(\{\widetilde{R}_j\}_{j\in[n]}; 1-\alpha''\right) \geq Q.\]
Next, defining
\[\widetilde{S}_{\textnormal{test}} = \sum_{i=n+1}^{n+m}\mathbf{1}\left\{\widetilde{R}_i > \textnormal{Quantile}\left(\{\widetilde{R}_j\}_{j\in[n]}; 1-\alpha''\right) \right\},\]
we therefore have
\[\widetilde{S}_{\textnormal{test}} \leq \sum_{i=n+1}^{n+m}\mathbf{1}\left\{\widetilde{R}_i > Q\right\} = A_{\textnormal{test}} < A \cdot \frac{m}{n+m} + \sqrt{\frac{m}{2}\log(1/\delta)}  \leq \alpha' m + \sqrt{\frac{m}{2}\log(1/\delta)} = \alpha''' m,\]
where $\alpha''' = \alpha' + \sqrt{\frac{\log(1/\delta)}{2m}}$ (again assuming the above events hold). Hence,
\[ \mathbb{P} \left\{\widetilde{S}_{\textnormal{test}} < \alpha''' m \right\} \ge 1-2\delta. \]

\subsubsection{Step 3: relating the jackknife+ to the oracle}
For a test point $i\in\{n+1,\dots,n+m\}$ and a training point $j\in[n]$, define 
\[R_{i;j} = |Y_i - \widehat\mu_{-j}(X_i)|.\]
Then, for each test point $i\in\{n+1,\dots,n+m\}$, we can verify that
\begin{align*}
    Y_i\not\in \widehat{C}_n^{\textnormal{$\gamma$-J+}}(X_i) & \Longrightarrow \ \sum_{j\in[n]}\mathbf{1}\{R_{i;j} > \gamma+R_j\}\geq (1-\alpha)(n+1) \\
    & \Longleftrightarrow \ \textnormal{Quantile}\big(\{\gamma+R_j - R_{i;j}\}_{j\in[n]}; (1-\alpha)(1+1/n)\big) < 0 \\ 
    & \ \Longrightarrow  \ \textnormal{Quantile}\big(\{\gamma+R_j - R_{i;j}\}_{j\in[n]}; 1-\alpha \big) < 0.
\end{align*}
Therefore,
\[S_{\textnormal{test}} = \sum_{i=n+1}^{n+m}\mathbf{1}\{Y_i\not\in \widehat{C}^{\textnormal{$\gamma$-J+}}_n(X_i)\} \leq \sum_{i=n+1}^{n+m}\mathbf{1}\left\{\textnormal{Quantile}\left( \{\gamma+ R_j - R_{i;j}\}_{j\in[n]};1-\alpha\right) < 0 \right\}.\]
And, recalling our definition of $\widetilde{S}_{\textnormal{test}}$, we can also write
\[\widetilde{S}_{\textnormal{test}} = \sum_{i=n+1}^{n+m}\mathbf{1}\left\{\widetilde{R}_i > \textnormal{Quantile}\left(\{\widetilde{R}_j\}_{j\in[n]};1-\alpha''\right)\right\}= \sum_{i=n+1}^{n+m}\mathbf{1}\left\{\textnormal{Quantile}\left( \{\widetilde{R}_j - \widetilde{R}_i\}_{j\in[n]};1-\alpha''\right) < 0 \right\}.\]
Our next task is to show that $S_{\textnormal{test}}$ is approximately bounded by $ \widetilde{S}_{\textnormal{test}}$, with high probability. We will show this by examining the differences in the terms that define $S_{\textnormal{test}}$ versus $\widetilde{S}_{\textnormal{test}}$. Fix any $i$, and suppose that
\[\mathbf{1}\left\{\textnormal{Quantile}\left( \{\gamma+ R_j - R_{i;j}\}_{j\in[n]};1-\alpha\right) < 0 \right\} = 1,\]
and
\[ \quad \mathbf{1}\left\{\textnormal{Quantile}\left( \{\widetilde{R}_j - \widetilde{R}_i\}_{j\in[n]};1-\alpha''\right) < 0 \right\} = 0.\]
For this to occur, we must have
\[\sum_{j\in[n]}\mathbf{1}\{\gamma + R_j - R_{i;j} < 0 \} \geq (1-\alpha)n, \quad \sum_{j\in[n]}\mathbf{1}\{\widetilde{R}_j - \widetilde{R}_i < 0 \} < (1-\alpha'')n.\]
In particular this implies
\begin{equation}\label{eqn:OG_difference}
    \sum_{j\in[n]} \mathbf{1}\left\{ \gamma + R_j - R_{i;j} < 0, \widetilde{R}_j - \widetilde{R}_i\geq 0 \right\} \geq (\alpha''-\alpha)n,
\end{equation}
which we can relax to
\begin{equation}\label{eqn:relax_difference}
    \sum_{j\in[n]} \mathbf{1}\left\{   (R_{i;j} - \widetilde{R}_i) - (R_j - \widetilde{R}_j) > \gamma\right\} \geq (\alpha''-\alpha)n.
\end{equation}
Therefore,
\[S_{\textnormal{test}}\leq \widetilde{S}_{\textnormal{test}} + \sum_{i=n+1}^{n+m} \mathbf{1}\left\{ \sum_{j\in[n]} \mathbf{1}\left\{   (R_{i;j} - \widetilde{R}_i) - (R_j - \widetilde{R}_j) > \gamma\right\} \geq (\alpha''-\alpha)n\right\}.\]

Now we apply $m$-stability.
By definition, 
for all $j\in[n]$ and $i\in\{n+1,\dots,n+m\}$, we have
\[\mathbb{E}\left[ | R_{i;j} - \widetilde{R}_i |\right] \leq 
\mathbb{E}\left[ | \widehat\mu_{-j}(X_i) - \widetilde{\mu}_{-i}(X_i)|\right] = \beta_{m,n-1}^{\textnormal{out}}(\mathcal{A},P),\]
since $\widehat\mu_{-j}$ is trained on $[n]\backslash\{j\}$ while $\widetilde{\mu}_{-i}$ is trained on $[n+m]\backslash\{i\}$ (i.e., the same training set, except with $m$ additional points), and similarly
\[\mathbb{E}\left[ | R_j - \widetilde{R}_j |\right] \leq \mathbb{E}\left[ | \widehat\mu_{-j}(X_j) - \widetilde{\mu}_{-j}(X_j) |\right] = \beta_{m,n-1}^{\textnormal{out}}(\mathcal{A},P),\]
since $\widehat\mu_{-j}$ is trained on $[n]\backslash\{j\}$ while $\widetilde{\mu}_{-j}$ is trained on $[n+m]\backslash\{j\}$. Therefore by Markov's inequality,
\begin{multline}\label{eqn:jack_key_step}
\mathbb{E}\left[\sum_{j\in[n]} \mathbf{1}\left\{   (R_{i;j} - \widetilde{R}_i) - (R_j - \widetilde{R}_j) > \gamma\right\} \right]
= \sum_{j\in[n]} \mathbb{P}\left\{   (R_{i;j} - \widetilde{R}_i) - (R_j - \widetilde{R}_j) > \gamma\right\}\\
\leq \sum_{j\in [n]}\frac{\mathbb{E}\left[  \left| (R_{i;j} - \widetilde{R}_i) - (R_j - \widetilde{R}_j)\right|\right]}{\gamma}\leq 
 n \cdot\frac{2\beta_{m,n-1}^{\textnormal{out}}(\mathcal{A},P)}{\gamma},
\end{multline}
and so applying Markov's inequality again,
\[\mathbb{P}\left\{\sum_{j\in[n]} \mathbf{1}\left\{   (R_{i;j} - \widetilde{R}_i) - (R_j - \widetilde{R}_j) > \gamma\right\} 
 \geq (\alpha'' - \alpha)n\right\} \leq \frac{2\beta_{m,n-1}^{\textnormal{out}}(\mathcal{A},P)}{\gamma(\alpha''-\alpha)},\] 
 as long as $\alpha'' > \alpha$.
Applying Markov's inequality one last time, for any $\delta' \in (0,1)$, with probability at least $1-\delta'$,
\[\sum_{i=n+1}^{n+m} \mathbf{1}\left\{ \sum_{j\in[n]} \mathbf{1}\left\{   (R_{i;j} - \widetilde{R}_i) - (R_j - \widetilde{R}_j) > \gamma\right\} \geq (\alpha''-\alpha)n\right\} \leq m \cdot \frac{2\beta_{m,n-1}^{\textnormal{out}}(\mathcal{A},P)}{\gamma(\alpha''-\alpha)}\cdot\frac{1}{\delta'}.\]
Combining the steps above, then, with probability at least $1-\delta'$,
\[S_{\textnormal{test}} \leq \widetilde{S}_{\textnormal{test}} +  m \cdot \frac{2\beta_{m,n-1}^{\textnormal{out}}(\mathcal{A},P)}{\gamma(\alpha''-\alpha)}\cdot\frac{1}{\delta'}.\] 

\subsubsection{Combining everything}
Combining all our steps so far, we have
\[\alpha_P^{\textnormal{$\gamma$-J+}}(\mathcal{D}_n) < \frac{S_{\textnormal{test}}}{m} + \sqrt{\frac{\log(1/\delta)}{2m}}, \ \widetilde{S}_{\textnormal{test}} < \alpha'''m , \ S_{\textnormal{test}} \leq \widetilde{S}_{\textnormal{test}} +  m \cdot \frac{2\beta_{m,n-1}^{\textnormal{out}}(\mathcal{A},P)}{\gamma(\alpha''-\alpha)}\cdot\frac{1}{\delta'},\]
with probability at least $1 - 3\delta - \delta'$.
This simplifies to
\[\alpha_P^{\textnormal{$\gamma$-J+}}(\mathcal{D}_n) < \alpha''' + \frac{2\beta_{m,n-1}^{\textnormal{out}}(\mathcal{A},P)}{\gamma(\alpha''-\alpha)}\cdot\frac{1}{\delta'} + \sqrt{\frac{\log(1/\delta)}{2m}}.\]
Recalling that
\[\alpha'' = \alpha' - \sqrt{\frac{\log(1/\delta)}{2n}}, \alpha''' = \alpha' + \sqrt{\frac{\log(1/\delta)}{2m}},\]
and choosing\footnote{Here we can assume $\alpha'\in[0,1]$ (as was required in Step 2 of the proof) since, if $\alpha'>1$, the the result of the theorem holds trivially.}
\[\alpha' = \alpha + \sqrt{\frac{\log(1/\delta)}{2n}} + \sqrt[3]{\frac{2\beta_{m,n-1}^{\textnormal{out}}(\mathcal{A},P)}{\gamma}}, \quad \delta' = \sqrt[3]{\frac{2\beta_{m,n-1}^{\textnormal{out}}(\mathcal{A},P)}{\gamma}},\]
we therefore have
\[\alpha_P^{\textnormal{$\gamma$-J+}}(\mathcal{D}_n) < \alpha + 3\sqrt{\frac{\log(1/\delta)}{2\min\{n,m\}}} + 2\sqrt[3]{\frac{2\beta_{m,n-1}^{\textnormal{out}}(\mathcal{A},P)}{\gamma}}\]
with probability at least $1 - 3\delta - \sqrt[3]{\frac{2\beta_{m,n-1}^{\textnormal{out}}(\mathcal{A},P)}{\gamma}}$, which completes the proof.

\subsection{Proof of Theorem~\ref{thm:jack_main_uninflated} (jackknife+ without inflation)}
This result for jackknife+ without inflation holds as an immediate corollary of Theorem~\ref{thm:jack_main}, which establishes training-conditional coverage for the inflated jackknife+ interval. 

To see why, first fix any $\gamma>0$.
Since $\widehat{C}^{\textnormal{J+}}_n(X)\subseteq \widehat{C}^{\textnormal{$\gamma$-J+}}_n(X)$, we have
\begin{multline*} \alpha_P^{\textnormal{J+}}(\mathcal{D}_n) = \mathbb{P}\{Y\not\in \widehat{C}^{\textnormal{J+}}_n(X)\mid \mathcal{D}_n\} = \mathbb{P}\{Y\not\in \widehat{C}^{\textnormal{$\gamma$-J+}}_n(X)\mid \mathcal{D}_n\}
+ \mathbb{P}\{Y\in \widehat{C}^{\textnormal{$\gamma$-J+}}_n(X)\backslash \widehat{C}^{\textnormal{J+}}_n(X)\mid \mathcal{D}_n\} \\= \alpha_P^{\textnormal{$\gamma$-J+}}(\mathcal{D}_n) 
 + \mathbb{P}\{Y\in \widehat{C}^{\textnormal{$\gamma$-J+}}_n(X)\backslash \widehat{C}^{\textnormal{J+}}_n(X)\mid \mathcal{D}_n\},\end{multline*}
 where probabilities are taken with respect to a new data point $(X,Y)\sim P$ sampled independently of the data set $\mathcal{D}_n$. Since we have defined $f_{Y|X}$ as the density of the conditional distribution of $Y\mid X$, we have
\[\mathbb{P}\{Y\in \widehat{C}^{\textnormal{$\gamma$-J+}}_n(X)\backslash \widehat{C}^{\textnormal{J+}}_n(X)\mid \mathcal{D}_n\} = \mathbb{E}\left[\int_{y\in \widehat{C}^{\textnormal{$\gamma$-J+}}_n(X)\backslash \widehat{C}^{\textnormal{J+}}_n(X)} f_{Y|X}(y\mid X) \;\mathsf{d}y \, \middle| \, \mathcal{D}_n\right],\]
where expected value is taken with respect to the marginal distribution of $X$. We can then bound this as
\begin{multline*}\mathbb{E}\left[\int_{y\in \widehat{C}^{\textnormal{$\gamma$-J+}}_n(X)\backslash \widehat{C}^{\textnormal{J+}}_n(X)} f_{Y|X}(y\mid X) \;\mathsf{d}y \, \middle| \, \mathcal{D}_n\right]
\leq \mathbb{E}\left[ \textnormal{Leb}\left(\widehat{C}^{\textnormal{$\gamma$-J+}}_n(X)\backslash \widehat{C}^{\textnormal{J+}}_n(X)\right) \cdot\sup_{y\in\mathbb{R}} f_{Y|X}(y\mid X)\,\middle|\,\mathcal{D}_n\right]\\ = 2\gamma\mathbb{E}\left[ \sup_{y\in\mathbb{R}} f_{Y|X}(y\mid X)\,\middle|\,\mathcal{D}_n\right] = 2\gamma\cdot B_{\textnormal{dens}},  \end{multline*}
by definition of $B_{\textnormal{dens}}$, and using the fact that the set difference $\widehat{C}^{\textnormal{$\gamma$-J+}}_n(X)\backslash \widehat{C}^{\textnormal{J+}}_n(X)$ must have Lebesgue measure $2\gamma$, by construction.
Therefore, the inequality
\[ \alpha_P^{\textnormal{J+}}(\mathcal{D}_n)  \leq  \alpha_P^{\textnormal{$\gamma$-J+}}(\mathcal{D}_n) + 2\gamma\cdot B_{\textnormal{dens}}\]
holds almost surely.

Applying the result of Theorem~\ref{thm:jack_main}, we then have 
\[\mathbb{P}\left\{ \alpha_P^{\textnormal{J+}}(\mathcal{D}_n) < \alpha + 3\sqrt{\frac{\log(1/\delta)}{2\min\{n,m\}}} + 2\sqrt[3]{\frac{2\beta_{m,n-1}^{\textnormal{out}}(\mathcal{A},P)}{\gamma}}  + 2\gamma\cdot B_{\textnormal{dens}}\right\}  \geq 1 - 3\delta - \sqrt[3]{\frac{2\beta_{m,n-1}^{\textnormal{out}}(\mathcal{A},P)}{\gamma}} ,\]
for any $\gamma>0$. Choosing $\gamma = \sqrt[4]{\frac{2\beta_{m,n-1}^{\textnormal{out}}(\mathcal{A},P)}{B_{\textnormal{dens}}^3}}$ completes the proof.

\subsection{Proof of Theorem~\ref{thm:CP_main} (full conformal)}
This proof is very similar to the proof of Theorem~\ref{thm:jack_main} for the jackknife+. We again consider a data set $(X_1,Y_1),\dots,(X_{n+m},Y_{n+m})$ drawn i.i.d.\ from $P$.

\subsubsection{Step 1: training-conditional coverage versus coverage on the test set}
Define
\[\widehat{C}^{\textnormal{$\gamma$-CP}}_n(x) = \left\{y \in\mathbb{R} : |y - \widehat{\mu}^y(x)| \leq \textnormal{Quantile}\big(\{|Y_i - \widehat{\mu}^y(X_i)|\}_{i\in[n]}; (1-\alpha)(1+1/n)\big) + \gamma\right\},\]
the $\gamma$-inflated full conformal prediction set trained on $\{(X_j,Y_j)\}_{j\in[n]}$ and evaluated at a test point $x$, and define
\[S_{\textnormal{test}} = \sum_{i=n+1}^{n+m}\mathbf{1}\{Y_i\not\in \widehat{C}^{\textnormal{$\gamma$-CP}}_n(X_i)\}.\]
Then, exactly as in Step 1 of the proof of Theorem~\ref{thm:jack_main}, we have
\[\mathbb{P}\left\{\alpha_P^{\textnormal{$\gamma$-CP}}(\mathcal{D}_n) < \frac{S_{\textnormal{test}}}{m} + \sqrt{\frac{\log(1/\delta)}{2m}}\right\} \geq 1-\delta.\]

\subsubsection{Step 2: define an oracle}
We now define the model fitted on the full data set,
\[\widetilde{\mu}=\mathcal{A}\left(\{(X_j,Y_j)\}_{j\in[n+m]}\right),\]
and its residuals
\[\widetilde{R}_j = |Y_j - \widetilde{\mu}(X_j)|, \ j\in[n+m].\]
By exchangeability of the data together with symmetry of $\mathcal{A}$, the residuals $\widetilde{R}_1,\dots,\widetilde{R}_{n+m}$ are exchangeable.
Define $\alpha'' = \alpha' - \sqrt{\log(1/\delta)/2n}$ and $\alpha''' = \alpha' +\sqrt{\log(1/\delta)/2m}$, and
\[\widetilde{S}_{\textnormal{test}} = \sum_{i=n+1}^{n+m}\mathbf{1}\left\{\widetilde{R}_i > \textnormal{Quantile}\left(\{\widetilde{R}_j\}_{j\in[n]}; 1-\alpha''\right) \right\}.\]
Following identical arguments as in Step 2 in the proof of Theorem~\ref{thm:jack_main}, we have
\[\widetilde{S}_{\textnormal{test}} < \alpha''' m,\]
with probability at least $1-2\delta$.

\subsubsection{Step 3: relating the full conformal interval to the oracle}
For each test point $i\in\{n+1,\dots,n+m\}$, define
\[\widehat\mu_{+i} = \mathcal{A}\left(\{(X_j,Y_j)\}_{j\in[n]\cup \{i\}}\right)\]
as the fitted model trained on the training data $j\in[n]$ along with test point $i$. Define also
\[R_{j;i} = |Y_j - \widehat\mu_{+i}(X_j)|, \ j\in[n],\]
and
\[R_i = |Y_i - \widehat\mu_{+i}(X_i)|.\]
Then by construction of the inflated full conformal interval, we can see that, for test point $i\in\{n+1,\dots,n+m\}$,
\begin{align*}Y_i\not\in \widehat{C}_n^{\textnormal{$\gamma$-CP}}(X_i) & \Longleftrightarrow \ R_i > \textnormal{Quantile}\left(\{R_{j;i}\}_{j\in[n]};(1-\alpha)(1+1/n)\right) + \gamma\\
& \Longleftrightarrow
\textnormal{Quantile}\left(\{\gamma + R_{j;i} - R_i\}_{j\in[n]};(1-\alpha)(1+1/n)\right)<0\\
& \Longrightarrow
\textnormal{Quantile}\left(\{\gamma + R_{j;i} - R_i\}_{j\in[n]};1-\alpha\right)<0.\end{align*}
Therefore,
\[S_{\textnormal{test}} = \sum_{i=n+1}^{n+m}\mathbf{1}\{Y_i\not\in \widehat{C}^{\textnormal{$\gamma$-CP}}_n(X_i)\} \leq \sum_{i=n+1}^{n+m}\mathbf{1}\left\{\textnormal{Quantile}\left( \{\gamma+ R_{j;i} - R_i\}_{j\in[n]};1-\alpha\right) < 0 \right\}.\]
And, recalling our definition of $\widetilde{S}_{\textnormal{test}}$, we can also write
\[\widetilde{S}_{\textnormal{test}} = \sum_{i=n+1}^{n+m}\mathbf{1}\left\{\widetilde{R}_i > \textnormal{Quantile}\left(\{\widetilde{R}_j\}_{j\in[n]};1-\alpha''\right)\right\} = \sum_{i=n+1}^{n+m}\mathbf{1}\left\{\textnormal{Quantile}\left( \{\widetilde{R}_j - \widetilde{R}_i\}_{j\in[n]};1-\alpha''\right) < 0 \right\}.\]
Now we apply $m$-stability.
By definition, 
for all $j\in[n]$ and $i\in\{n+1,\dots,n+m\}$, we have
\begin{equation}\label{eqn:CP_key_step_a}\mathbb{E}\left[ | R_{j;i} - \widetilde{R}_j |\right] \leq 
\mathbb{E}\left[ | \widehat\mu_{+i}(X_j) - \widetilde{\mu}(X_j)|\right] = \beta^{\textnormal{in}}_{m-1,n+1}(\mathcal{A},P),\end{equation}
since $\widehat\mu_{+i}$ is trained on $[n]\cup\{i\}$ while $\widetilde{\mu}$ is trained on $[n+m]$ (i.e., the same training set, except with $m-1$ additional points), and similarly
\begin{equation}\label{eqn:CP_key_step_b}\mathbb{E}\left[ | R_i - \widetilde{R}_i |\right] \leq \mathbb{E}\left[ | \widehat\mu_{+i}(X_i) - \widetilde{\mu}(X_i) |\right] = \beta^{\textnormal{in}}_{m-1,n+1}(\mathcal{A},P).\end{equation}
Hence, as for jackknife+, Markov's inequality implies
\begin{equation}\label{eqn:CP_key_step}\mathbb{E}\left[\sum_{j\in[n]} \mathbf{1}\left\{   (\widetilde{R}_j - R_{j;i}) + (R_i - \widetilde{R}_i) > \gamma\right\} \right] \leq n \cdot \frac{2\beta_{m-1,n+1}^{\textnormal{in}}(\mathcal{A},P)}{\gamma}.\end{equation}
From this point on, following identical arguments as in Step 3 of the proof of Theorem~\ref{thm:jack_main}, we obtain 
\[S_{\textnormal{test}} \leq \widetilde{S}_{\textnormal{test}} +  m \cdot \frac{2\beta_{m-1,n+1}^{\textnormal{in}}(\mathcal{A},P)}{\gamma(\alpha''-\alpha)}\cdot\frac{1}{\delta'}\]
with probability at least $1-\delta'$.

\subsubsection{Combining everything}
Following the same steps, and analogous choices of $\alpha',\delta'$, as in the proof of Theorem~\ref{thm:jack_main}, we obtain
\[\alpha_P^{\textnormal{$\gamma$-CP}}(\mathcal{D}_n) < \alpha + 3\sqrt{\frac{\log(1/\delta)}{2\min\{n,m\}}} + 2\sqrt[3]{\frac{2\beta_{m-1,n+1}^{\textnormal{in}}(\mathcal{A},P)}{\gamma}}\]
with probability at least $1 - 3\delta - \sqrt[3]{\frac{2\beta_{m-1,n+1}^{\textnormal{in}}(\mathcal{A},P)}{\gamma}}$, which completes the proof.

\section{Discussion}\label{sec:discussion}

Algorithmic stability is a key ingredient in many statistical frameworks, and in particular, has played an important role in distribution-free predictive inference, allowing for stronger guarantees for methods such as jackknife and jackknife+. In this paper, we have seen that a version of the stability assumption---specifically, stability with respect to adding/removing (or swapping) $m$ training points relative to a data set of size $n$---is sufficient to ensure training-conditional coverage for predictive intervals constructed by the full conformal method or by jackknife+, thus establishing an assumption-lean setting in which these popular methods offer the important training-conditional coverage property.

The results of this work suggest a number of potential directions for further study. First, one remaining question is the problem of understanding the $m$-stability properties of commonly used algorithms: while our results show that $m$-stability can be (coarsely) bounded with existing stability definitions, our example calculations and empirical results show a more optimistic picture, suggesting that  algorithms used in practice may have favorable $m$-stability properties that are not yet well understood.  Second, as mentioned in Section~\ref{sec:background}, in this paper we have studied only one specific version of the split and full conformal algorithms: namely, prediction intervals (or sets) constructed via the residual, $|y - \hat\mu(x)|$, where $\hat\mu$ represents a fitted regression model. The conformal prediction framework is able to use a much broader range of constructions (``nonconformity scores'' \citep{vovk2005algorithmic}), to adapt to different types of data distributions---for example, the Conformalized Quantile Regression (CQR) method of \citet{romano2019conformalized}. An important question is whether notions of algorithmic stability can be extended more broadly to different conformal prediction constructions, to ensure training-conditional coverage beyond the residual score.
We leave these open questions for future work.

\subsection*{Acknowledgements}
The authors thank the Associate Editor and anonymous referees for their valuable comments. The authors also thank Michael Bian for helpful feedback on an earlier draft of this manuscript. 
R.F.B.\ was partially supported by the Office of Naval Research via grant N00014-20-1-2337, and by the National Science Foundation via grant DMS-2023109.

\bibliographystyle{plainnat}
\bibliography{bib}

\appendix

\section{Proof of asymptotic results}\label{appA_asy_proof}
We will now see how Lemma~\ref{lem:m_vs_1} leads to proofs of the asymptotic results, Theorems~\ref{thm:jack_asymp} and~\ref{thm:CP_asymp}. In particular, while our finite-sample results are phrased in terms of $m$-stability, this lemma enables us to verify that, asymptotically, $1$-stability is sufficient for training-conditional coverage.

\begin{proof}[Proof of Theorems~\ref{thm:jack_asymp} and~\ref{thm:CP_asymp}]
First, we will show that the first part of Theorem~\ref{thm:jack_asymp} (i.e., for the inflated jackknife+ interval) follows as a consequence of Theorem~\ref{thm:jack_main} combined with Lemma~\ref{lem:m_vs_1}. 
 
First, define $m_n  \leq  \left\lceil  (b_{n}/\gamma_{n})^{-1/2} \right\rceil $, and note that $m_n\to\infty$ and $m_nb_n/\gamma_n\to 0$ since $b_{n}/\gamma_{n}\to 0$.
Define also
\[\epsilon_n =  \frac{3}{\sqrt[4]{4\min\{n,m_n\}}}+ 2\sqrt[3]{\frac{2 m_nb_{n}}{\gamma_n}}, \quad \delta_n = 3e^{-\sqrt{\min\{n,m_n\}}} + \sqrt[3]{\frac{2 m_nb_{n}}{\gamma_n}}.\]
Observe that $\epsilon_n,\delta_n\rightarrow 0$, since $m_n\to\infty$ and $m_nb_n/\gamma_n\to 0$. 

Next, for any algorithm $\mathcal{A}$ and any distribution $P$, if $\beta^{*,\textnormal{out}}_{1,n-1}(\mathcal{A},P)\leq b_n$, then applying Lemma~\ref{lem:m_vs_1} with $m=m_n$, we have
\[\beta_{m_n,n-1}^{\textnormal{out}}(\mathcal{A}, P)
\leq \sum_{k=n-1}^{n+m_n-2}\beta_{1,k}^{\textnormal{out}}(\mathcal{A}, P)
\leq m_nb_{n} .\]
Then applying Theorem~\ref{thm:jack_main} with $\delta = e^{-\sqrt{\min\{n,m_n\}}}$, and plugging in our choices of $m_n,\gamma_n$ and our calculations above, we have 
\[  \mathbb{P}\left\{ \alpha_{P}^{\textnormal{$\gamma_n$-J+}}(\mathcal{D}_n)\geq \alpha + \epsilon_n \right\} \leq \delta_n,\]
for all such $\mathcal{A},P$. In other words, we have shown that
\[\sup \limits_{\substack{\mathcal{A}, P\textnormal{ s.t. }\\ \beta_{1,n-1}^{*,\textnormal{out}}(\mathcal{A},P) \le b_{n}}}\mathbb{P}\left\{\alpha_{P}^{\textnormal{$\gamma_n$-J+}}(\mathcal{D}_n)\geq \alpha + \epsilon_n\right\}  \leq \delta_n.\]
This completes the proof of the first part of Theorem~\ref{thm:jack_asymp}, i.e., asymptotic coverage for the inflated jackknife+ interval.

We now proceed to verify the second claim in Theorem~\ref{thm:jack_asymp}, i.e., for the uninflated jackknife+ interval; this will follow from a very similar argument as for the first claim. 

First, define $m_n = \lceil (b_n B_n)^{-1/2}\rceil$, and note that $m_n\to\infty$ and $m_nb_nB_n\to 0$ since $b_{n}B_n\to 0$.
Define also
\[\epsilon_n =  \frac{3}{\sqrt[4]{4\min\{n,m_n\}}}+ 4\sqrt[4]{2B_n \cdot m_n b_n}, \quad \delta_n = 3e^{-\sqrt{\min\{n,m_n\}}} + \sqrt[4]{2B_n \cdot m_n b_n}.\]
Observe that $\epsilon_n,\delta_n\rightarrow 0$, since $m_n\to\infty$ and $m_nb_nB_n\to 0$. 

Next, for any algorithm $\mathcal{A}$ and any distribution $P\in\mathcal{P}_{B_n}$, if $\beta^{*,\textnormal{out}}_{1,n-1}(\mathcal{A},P)\leq b_n$, then applying Lemma~\ref{lem:m_vs_1}  we have $\beta_{m_n,n-1}^{\textnormal{out}}(\mathcal{A}, P)
\leq m_nb_{n}$, as before. Therefore, by Theorem~\ref{thm:jack_main_uninflated} with $\delta = e^{-\sqrt{\min\{n,m_n\}}}$,
\[\mathbb{P}\left\{ \alpha_{P}^{\textnormal{J+}}(\mathcal{D}_n)  \ge \alpha + \epsilon_n\right\} \le \delta_n.\]
Since this holds for all such $\mathcal{A},P$, in other words we have shown that
\[\sup_{\substack{\mathcal{A}, P\in\mathcal{P}_{B_n}\textnormal{ s.t.}\\\beta^{*,\textnormal{out}}_{1,n-1}(\mathcal{A},P)\leq b_n}}\mathbb{P}\left\{ \alpha_{P}^{\textnormal{J+}}(\mathcal{D}_n)  \ge \alpha + \epsilon_n\right\} \le \delta_n,\]
which completes the proof of Theorem~\ref{thm:jack_asymp}.

For full conformal, the proof of Theorem~\ref{thm:CP_asymp} follows via an identical argument as that for the jackknife+ with inflation, i.e., by combining the finite-sample result of Theorem~\ref{thm:CP_main}  with Lemma~\ref{lem:m_vs_1}; we omit the details.
\end{proof}

\section{Proofs of training-conditional coverage under alternative stability frameworks}\label{appB}
\subsection{Proof of Theorem~\ref{thm:jack_version_P} (jackknife+---alternative version)}
First we prove the first part of Theorem~\ref{thm:jack_version_P}, for the inflated jackknife+ interval. This proof is identical to the proof of Theorem~\ref{thm:jack_main}, which uses the expected value definition of out-of-sample stability, except for one modification. Specifically, for the step~\eqref{eqn:jack_key_step} in the proof of Theorem~\ref{thm:jack_main}, we instead have the bound
\begin{multline*}\mathbb{E}\left[\sum_{j\in[n]} \mathbf{1}\left\{   (R_{i;j} - \widetilde{R}_i) - (R_j - \widetilde{R}_j) > 2\epsilon_{m,n-1}\right\} \right] = \sum_{j\in[n]} \mathbb{P}\left\{   (R_{i;j} - \widetilde{R}_i) - (R_j - \widetilde{R}_j) > 2\epsilon_{m,n-1}\right\}  \\
\leq \sum_{j\in[n]} \mathbb{P}\left\{   |R_{i;j} - \widetilde{R}_i|>\epsilon_{m,n-1}\right\} + \mathbb{P}\left\{ |R_j - \widetilde{R}_j| > \epsilon_{m,n-1}\right\} \leq n \cdot 2\nu_{m,n-1},\end{multline*}
by the stability assumption of Theorem~\ref{thm:jack_version_P}. The rest of the proof is identical to the proof of Theorem~\ref{thm:jack_main}.

Next we turn to the second part of the theorem, proving training-conditional coverage for the uninflated jackknife+ interval. This result will follow as a corollary of the first part of the theorem. First, we have
\[\alpha_P^{\textnormal{J+}}(\mathcal{D}_n)\leq 
\alpha_P^{\textnormal{$2\epsilon_{m,n-1}$-J+}}(\mathcal{D}_n) + 4\epsilon_{m,n-1}\cdot B_{\textnormal{dens}},\]
by the same argument as in the proof of Theorem~\ref{thm:jack_main_uninflated} (i.e., with inflation term $2\epsilon_{m,n-1}$ in place of $\gamma$). Combining this inequality with the first part of the theorem completes the proof of the second claim.

\subsection{Proof of Theorem~\ref{thm:CP_version_P} (full conformal---alternative version)}
This proof is identical to the proof of Theorem~\ref{thm:CP_main}, which uses the expected value definition of in-sample stability, except for one modification. Specifically, for the step~\eqref{eqn:CP_key_step} in the proof of Theorem~\ref{thm:CP_main}, we instead have the bound
\begin{multline*}\mathbb{E}\left[\sum_{j\in[n]} \mathbf{1}\left\{ (\widetilde{R}_j - R_{j;i}) + (R_i - \widetilde{R}_i) >2\epsilon_{m-1,n+1}  \right\} \right] = \sum_{j\in[n]} \mathbb{P}\left\{  (\widetilde{R}_j - R_{j;i}) + (R_i - \widetilde{R}_i) >2\epsilon_{m-1,n+1} \right\}  \\
\leq \sum_{j\in[n]} \mathbb{P}\left\{   |R_{j;i} - \widetilde{R}_j|>\epsilon_{m-1,n+1}\right\} + \mathbb{P}\left\{ |R_i - \widetilde{R}_i| > \epsilon_{m-1,n+1}\right\}  \leq n \cdot 2\nu_{m-1,n+1},\end{multline*}
by the stability assumption of Theorem~\ref{thm:CP_version_P}. The rest of the proof is identical.

\subsection{Proof of Theorem~\ref{thm:jack_swap} (jackknife+ with swap-stability)}\label{appB_proof_jack_swap}
As mentioned in Section~\ref{sec:swap_preview}, the swap-stability is sufficient for jackknife+ to achieve training-conditional coverage because a weaker version of converse to Proposition~\ref{prop:stable_implies_swap_stable} can be established for out-of-sample stability.
To state the partial converse result, we introduce some notations. Recalling that an algorithm $\mathcal{A}$ maps \emph{a data set of any size} to a fitted function, we will write $\mathcal{A}_n$ to denote the map defined by applying $\mathcal{A}$ to a training set of size $n$. That is, $\mathcal{A}_n$ is a map with input space $(\mathcal{X}\times\mathbb{R})^n$, and $\mathcal{A}$ is simply the collection of maps $\mathcal{A}_n$ across all $n$.
\begin{proposition}\label{prop:swap_stability_implies_stability}
For any $\mathcal{A}$ and any $P$, for any $n$ there exists an algorithm $\mathcal{A}'$ with $\mathcal{A}'_n = \mathcal{A}_n$, such that
\[\beta_{m,n}^{\textnormal{out}}(\mathcal{A}',P) \leq \bar{\beta}_{m,n}^{\textnormal{out}}(\mathcal{A},P).\]
\end{proposition}
\begin{proof}[Proof of Proposition~\ref{prop:swap_stability_implies_stability}]
Define the algorithm $\mathcal{A}'$ by setting $\mathcal{A}'_k = \mathcal{A}_k$ for all $k\neq n+m$, and defining $\mathcal{A}'_{n+m}$ as follows: given a data set $(X_1,Y_1),\dots,(X_{n+m},Y_{n+m})$, for each subset $S\subseteq[n+m]$ with $|S|=n$, let $\widehat\mu_S = \mathcal{A}_n(\{(X_j,Y_j)\}_{j\in S})$. Then $\mathcal{A}'_{n+m}$ returns the map $\widehat\mu_{n+m}$ given by
\[\widehat\mu_{n+m}(x) = \frac{1}{{n+m\choose n}}\sum_{S\subseteq[n+m],|S|=n} \widehat\mu_S(x).\]
In other words, $\mathcal{A}'_{n+m}$ operates by taking an expected value over running $\mathcal{A}_n$ on a random subset of size $n$.

Now we have
\begin{multline*}\beta^{\textnormal{out}}_{m,n}(\mathcal{A}',P) = \mathbb{E}\left[
\left|\widehat\mu_n(X) - \widehat\mu_{n+m}(X)\right|\right]
= \mathbb{E}\left[
\left|\widehat\mu_n(X) - \frac{1}{{n+m\choose n}}\sum_{S\subseteq[n+m],|S|=n} \widehat\mu_S(X)\right|\right]\\
\leq \frac{1}{{n+m\choose n}}\sum_{S\subseteq[n+m],|S|=n}\mathbb{E}\left[ \left|\widehat\mu_n(X) 
- \widehat\mu_S(X)\right|\right].
\end{multline*}
Next, for each $S$, we need to bound $\mathbb{E}\left[ \left|\widehat\mu_n(X) 
- \widehat\mu_S(X)\right|\right]$. Let $k = |S \cap \{n+1,\dots,n+m\}|$, then $0 \le k \le m$. Without loss of generality, $S = \{1,\dots,n-k, n+1,\dots,n+k\}$. Then $\mathbb{E}[|\widehat\mu_n(X) - \widehat\mu_S(X)] = \bar\beta^{\textnormal{out}}_{k,n}(\mathcal{A},P)\leq \bar\beta^{\textnormal{out}}_{m,n}(\mathcal{A},P)$, where the last step holds for $k\leq m$ by applying Proposition~\ref{prop:beta_swap_order}, which is given below. This completes the proof.
\end{proof}

We are now ready to see how this partial converse reveals that swap-stability is sufficient for training-conditional coverage of the jackknife+.
\begin{proof}[Proof of Theorem~\ref{thm:jack_swap} (jackknife+ with swap-stability)]
First, we can observe that $\widehat{C}_n^{\textnormal{$\gamma$-J+}}(X_{n+1})$ depends on the algorithm $\mathcal{A}$ only through the map $\mathcal{A}_{n-1}$; that is, when running jackknife+, $\mathcal{A}$ is only ever trained on data sets of size $n-1$. Now let $\mathcal{A}'$ be defined as in Proposition~\ref{prop:swap_stability_implies_stability} (except with $n-1$ in place of $n$), so that we have $\mathcal{A}'_{n-1} = \mathcal{A}_{n-1}$. In other words, at sample size $n-1$, running jackknife+ with $\mathcal{A}$ is exactly equivalent to running jackknife+ with algorithm $\mathcal{A}'$. Therefore, the results of Theorems~\ref{thm:jack_main} and~\ref{thm:jack_main_uninflated} can be applied with $\mathcal{A}'$ in place of $\mathcal{A}$; applying Proposition~\ref{prop:swap_stability_implies_stability} tells us that $\beta_{m,n-1}^{\textnormal{out}}(\mathcal{A}',P) \leq \bar{\beta}_{m,n-1}^{\textnormal{out}}(\mathcal{A},P)$, which verifies the desired finite-sample training-coverage guarantee. Finally, the asymptotic guarantees with swap-stability are proved exactly as for (non-swap) stability, i.e., following the same arguments as in the proof of Theorem~\ref{thm:jack_asymp}.
\end{proof}

\subsection{Proof of Theorem~\ref{thm:CP_swap} (full conformal with swap-stability)}\label{sec:proof_thm:CP_swap}
Next we prove that swap-stability is sufficient to achieve finite-sample and asymptotic training-conditional coverage for full conformal. 
As for jackknife+, the asymptotic guarantee is proved exactly as before, by combining the finite-sample guarantee (which we will prove next) with Lemma~\ref{lem:m_vs_1}.

For the finite-sample guarantee, on the other hand, the proof for jackknife+, which relied on the ``weak converse'' relating out-of-sample stability to out-of-sample swap-stability (Proposition~\ref{prop:swap_stability_implies_stability}), cannot be reproduced for full conformal, because the ``weak converse'' result may no longer hold for in-sample stability; we will need to take a different approach.

To prove the finite-sample guarantee, we modify the proof of Theorem~\ref{thm:CP_main} by taking a different definition of the oracle residuals, $\widetilde{R}_j$: we define
\[\widetilde{R}_j = \left|Y_j - \frac{1}{{n+m-1\choose n}}\sum_{S\subseteq[n+m], |S|=n+1, j\in S} \widetilde{\mu}_S(X_j)\right|, \ j\in[n+m],\]
where for any subset $S\subseteq[n+m]$ we define $\widetilde{\mu}_S = \mathcal{A}\big(\{(X_j,Y_j)\}_{j\in S}\big)$. Exchangeability of the data, along with symmetry of $\mathcal{A}$, ensures exchangeability of $(\widetilde{R}_1,\dots,\widetilde{R}_{n+m})$, as before.

From this point on, the proof proceeds exactly as for Theorem~\ref{thm:CP_main}, except that in place of the bounds~\eqref{eqn:CP_key_step_a} and~\eqref{eqn:CP_key_step_b}, we instead need to verify the bounds
\[\mathbb{E}\left[ | R_{j;i} - \widetilde{R}_j |\right] 
 \leq \bar\beta^{\textnormal{in}}_{m-1,n+1}(\mathcal{A},P), \ j\in[n],\]
and
\[\mathbb{E}\left[ | R_i - \widetilde{R}_i |\right] \leq \bar\beta^{\textnormal{in}}_{m-1,n+1}(\mathcal{A},P), \ i\in\{n+1,\dots,n+m\}.\]
For the first bound, for $j\in [n]$, we have
\begin{multline*}\mathbb{E}\left[ | R_{j;i} - \widetilde{R}_j |\right]
= \mathbb{E}\left[ \left| |Y_j - \widehat\mu_{+i}(X_j)| - \left|Y_j - \frac{1}{{n+m-1\choose n}}\sum_{S\subseteq[n+m],|S|=n+1,j\in S} \widetilde{\mu}_S(X_j)\right|\right|\right]\\
\leq \mathbb{E}\left[ \left| \widehat\mu_{+i}(X_j) -  \frac{1}{{n+m-1\choose n}}\sum_{S\subseteq[n+m],|S|=n+1,j\in S} \widetilde{\mu}_S(X_j)\right|\right]\\\leq \frac{1}{{n+m-1\choose n}}\sum_{S\subseteq[n+m],|S|=n+1,j\in S} \mathbb{E}\left[ \left| \widehat\mu_{+i}(X_j) -   \widetilde{\mu}_S(X_j)\right|\right]\leq \bar\beta^{\textnormal{in}}_{m-1,n+1}(\mathcal{A},P),
\end{multline*}
where the last step holds due to Proposition~\ref{prop:beta_swap_order}. The second bound is proved similarly.

\section{Proofs for $m$-stability calculations}
\subsection{Proof of Proposition~\ref{prop:beta_swap_order} (monotonicity of swap-stability in $m$)}\label{sec:proof_prop:beta_swap_order}
The proof of the proposition relies on the following lemma:
\begin{lemma}\label{lem:L1_lemma}
Let $A,B,C$ be mutually independent random variables, and let $(A',B',C')$ be an i.i.d.\ copy of $(A,B,C)$. Then for any function $f$,
\[\mathbb{E}\left[|f(A,B,C) - f(A,B,C')|\right] \leq \mathbb{E}\left[|f(A,B,C) - f(A,B',C')|\right].\]
\end{lemma}
Define models
\begin{align*}
    \widehat\mu_n &= \mathcal{A}\left((X_1,Y_1),\dots,(X_n,Y_n)\right),\\
    \widehat\mu_n' &= \mathcal{A}\left((X_1,Y_1),\dots,(X_{n-m},Y_{n-m}),(X'_{n-m+1},Y'_{n-m+1}),\dots,(X'_n,Y'_n)\right),\\
    \widehat\mu_n'' &= \mathcal{A}\left((X_1,Y_1),\dots,(X_{n-m'},Y_{n-m'}),(X'_{n-m'+1},Y'_{n-m'+1}),\dots,(X'_n,Y'_n)\right),
\end{align*}
where all data points are drawn i.i.d.\ from $P$. Then by definition we have
\[\bar\beta^{\textnormal{out}}_{m,n}(\mathcal{A},P)  = \mathbb{E}\left[|\widehat\mu_n(X) - \widehat\mu_n'(X)|\right], \quad \bar\beta^{\textnormal{out}}_{m',n}(\mathcal{A},P)  = \mathbb{E}\left[|\widehat\mu_n(X) - \widehat\mu_n''(X)|\right]. \]
Now define 
\begin{align*}
    A &= \big((X_1,Y_1),\dots,(X_{n-m'},Y_{n-m'})\big),\\
    B &= \big((X_{n-m'+1},Y_{n-m'+1}),\dots,(X_{n-m},Y_{n-m})\big),\\
    C &= \big((X_{n-m+1},Y_{n-m+1}),\dots,(X_n,Y_n)\big).
\end{align*}
Then, following the notation of Lemma~\ref{lem:L1_lemma}, we can write 
\[\widehat\mu_n(X) = f(A,B,C),\quad \widehat\mu'_n(X) = f(A,B,C'), \quad \widehat\mu''_n(X) = f(A,B',C'),\]
for the appropriately defined function $f$ (i.e., $f$ is the function that applies the algorithm $\mathcal{A}$ to all $n$ data pairs contained in $A,B,C$, and then evaluates the resulting fitted model on an independent copy of the data $(X,Y)$). Applying the lemma immediately verifies that $\bar\beta^{\textnormal{out}}_{m,n}(\mathcal{A},P)\leq \bar\beta^{\textnormal{out}}_{m',n}(\mathcal{A},P)$.

For in-sample stability, we instead have
\[\bar\beta^{\textnormal{in}}_{m,n}(\mathcal{A},P)  = \mathbb{E}\left[|\widehat\mu_n(X_1) - \widehat\mu_n'(X_1)|\right], \quad \bar\beta^{\textnormal{in}}_{m',n}(\mathcal{A},P)  = \mathbb{E}\left[|\widehat\mu_n(X_1) - \widehat\mu_n''(X_1)|\right]. \]
Choosing the same $A,B,C$, we now have
\[\widehat\mu_n(X_1) = f(A,B,C),\quad \widehat\mu'_n(X_1) = f(A,B,C'), \quad \widehat\mu''_n(X_1) = f(A,B',C'),\]
for an appropriate new definition of the function $f$ (i.e., $f$ is the function that applies the algorithm $\mathcal{A}$ to all $n$ data pairs contained in $A,B,C$, and then evaluates the resulting fitted model on the first value of $A$). Applying the lemma  verifies that $\bar\beta^{\textnormal{in}}_{m,n}(\mathcal{A},P)\leq \bar\beta^{\textnormal{in}}_{m',n}(\mathcal{A},P)$.

\begin{proof}[Proof of Lemma~\ref{lem:L1_lemma}]
We will show the stronger statement
\[\mathbb{E}\left[|f(A,B,C) - f(A,B,C')| \mid A\right] \leq \mathbb{E}\left[|f(A,B,C) - f(A,B',C')| \mid A\right]\textnormal{ almost surely}.\]
Equivalently, we will show
\[\mathbb{E}\left[|f(a,B,C) - f(a,B,C')|\right] \leq \mathbb{E}\left[|f(a,B,C) - f(a,B',C')| \right]\]
for all $a$.

First, let $g_a:[0,1]\rightarrow\mathbb{R}$ be a left-continuous monotone nondecreasing function such that $g_a(U)$ is equal in distribution to  $f(a,B,C)$, where $U\sim\textnormal{Unif}[0,1]$---that is, $g_a$ is the inverse CDF of the random variable $f(a,B,C)$.
Then we have
\[\mathbb{E}\left[|f(a,B,C) - f(a,B',C')|\right] = \mathbb{E}\left[|g_a(U)-g_a(U')|\right],\]
where $U,U'\stackrel{\textnormal{iid}}{\sim}\textnormal{Unif}[0,1]$.
Next,  we can calculate
\begin{multline*}
    \mathbb{E}\left[|g_a(U)-g_a(U')|\right]
= \mathbb{E}\left[(g_a(U) - g_a(U'))\cdot\textnormal{sign}(g_a(U)-g_a(U'))\right] \\
= \mathbb{E}\left[(g_a(U) - g_a(U'))\cdot (\mathbf{1}_{U>U'} - \mathbf{1}_{U'>U})\right]
= 2\mathbb{E}\left[g_a(U) \cdot (\mathbf{1}_{U>U'} - \mathbf{1}_{U'>U})\right]
\\
= 2\mathbb{E}[g_a(U)\cdot \mathbb{E}[\mathbf{1}_{U>U'} - \mathbf{1}_{U'>U} \mid U]]
= 2\mathbb{E}[g_a(U)\cdot (2U - 1)].\end{multline*}
Combining our calculations so far, then, we have shown that
\[\mathbb{E}\left[|f(a,B,C) - f(a,B',C')|\right] =2\mathbb{E}[g_a(U)\cdot(2U-1)] = 4\mathbb{E}[g_a(U)\cdot U] - 2\mathbb{E}[f(a,B,C)],\]
where the last step holds since $g_a(U)$ is equal in distribution to $f(a,B,C)$, by construction.

Next, for each $a,b$, let $h_{a,b}[0,1]\rightarrow\mathbb{R}$ be a left-continuous monotone nondecreasing function such that $h_{a,b}(U)$ is equal in distribution to  $f(a,b,C)$, where $U\sim\textnormal{Unif}[0,1]$---that is, $h_{a,b}$ is the inverse CDF of the random variable $f(a,b,C)$.
Following identical calculations, we obtain
\[\mathbb{E}\left[|f(a,b,C) - f(a,b,C')|\right] = 2\mathbb{E}[h_{a,b}(U)\cdot(2U-1)].\]
Therefore, it holds almost surely over $B$ that
\[\mathbb{E}\left[|f(a,B,C) - f(a,B,C')| \mid B\right] = 2\mathbb{E}[h_{a,B}(U)\cdot(2U-1) \mid B],\]
and after marginalizing over $B$, then,
\[\mathbb{E}\left[|f(a,B,C) - f(a,B,C')| \right] = 2\mathbb{E}[h_{a,B}(U)\cdot(2U-1) ] = 4\mathbb{E}[h_{a,B}(U)\cdot U] - 2 \mathbb{E}[f(a,B,C)],\]
where the last step holds since $h_{a,B}(U)$ is equal in distribution to $f(a,B,C)$, by construction.

To complete the proof we then need to verify that
\[\mathbb{E}[h_{a,B}(U)\cdot U ] \leq \mathbb{E}[g_a(U)\cdot U].\]
As above, $g_a(U)$ and $h_{a,B}(U)$ are each equal in distribution to $f(a,B,C)$, and therefore are equal in distribution to each other. However, $g_a$ is also a monotone nondecreasing function. The desired bound then holds by the rearrangement inequality, which tells us that, for any random variables $(X,Y)$ and any monotone nondecreasing functions $f,g$ such that $f(Z),g(Z)$ are equal in distribution to $X,Y$, respectively, for some random variable $Z$, it holds that $\mathbb{E}[XY]\leq \mathbb{E}[f(Z)g(Z)]$.\end{proof}

\subsection{Proof of Lemma~\ref{lem:m_vs_1} (relating $1$-stability and $m$-stability)}
Let $(X,Y),(X_1,Y_1),$ $(X_2,Y_2),\dots,(X_{n+m},Y_{n+m})$ be i.i.d.\ samples drawn from $P$.
Define fitted model $\widehat\mu_k = \mathcal{A}((X_1,Y_1),\dots,(X_k,Y_k))$, for each $k=n,n+1,\dots,n+m$. Then we have
\[\beta_{m,n}^{\textnormal{out}}(\mathcal{A},P)
= \mathbb{E}\left[\left|\widehat\mu_n(X) - \widehat\mu_{n+m}(X)\right|\right] \leq \sum_{k=n}^{n+m-1}\mathbb{E}\left[\left|\widehat\mu_k(X) - \widehat\mu_{k+1}(X)\right|\right] = \sum_{k=n}^{n+m -1} \beta_{1,k}^{\textnormal{out}}(\mathcal{A},P),\]
by definition of out-of-sample $1$-stability. A similar argument holds for in-sample stability as well.

Next, we prove the analogous results for swap-stability. For each $k=0,\dots,m$, we define  the fitted model $\widehat\mu_n^{(k)} = \mathcal{A}((X_1,Y_1),\dots,(X_{n-k},Y_{n-k}),(X'_{n-k+1},Y'_{n-k+1}),\dots,(X'_n,Y'_n))$, where the data points $(X,Y),(X_1,Y_1),\dots,(X_n,Y_n),(X'_{n-m+1},Y'_{n-m+1}),\dots,(X'_n,Y'_n)$ are i.i.d.\ samples drawn from $P$. Then we have
\[\bar{\beta}_{m,n}^{\textnormal{out}}(\mathcal{A},P)
= \mathbb{E}\left[\left|\widehat\mu_n^{(0)}(X) - \widehat\mu_n^{(m)}(X)\right|\right]\leq \sum_{k=1}^m\mathbb{E}\left[\left|\widehat\mu_n^{(k-1)}(X) - \widehat\mu_n^{(k)}(X)\right|\right] = \sum_{k=1}^m \bar{\beta}_{1,n}^{\textnormal{out}}(\mathcal{A},P) = m\bar{\beta}_{1,n}^{\textnormal{out}}(\mathcal{A},P),\]
by definition of out-of-sample $1$-swap-stability together with the assumption of symmetry of $\mathcal{A}$. Again, a similar argument holds for in-sample swap-stability.

\subsection{Proof of Proposition~\ref{prop:kNN} ($k$-nearest neighbors)}
For simplicity, we will assume that ties (for determining the set of nearest neighbors) do not occur, almost surely, for data drawn from $P$ (e.g., $X$ might be continuously distributed); a similar argument would hold in the case of randomly breaking ties, for the general case (if we allow $\mathcal{A}$ to be randomized in order to enable tiebreaking). 

We calculate
\begin{align*}
    \beta^{\textnormal{out}}_{m,n}(\mathcal{A},P)
    &=\mathbb{E}\left[|\widehat\mu_{n+m}(X) - \widehat\mu_n(X)|\right]\\
    &=\mathbb{E}\left[\left|\frac{1}{k}\sum_{i\in N_{k,n+m}(X)}Y_i -\frac{1}{k}\sum_{i\in N_{k,n}(X)}Y_i\right| \right]\\
    &=\mathbb{E}\left[\left|\frac{1}{k}\sum_{i\in N_{k,n+m}(X)\backslash N_{k,n}(X)}Y_i -\frac{1}{k}\sum_{i\in N_{k,n}(X)\backslash N_{k,n+m}(X)}Y_i\right| \right]\\
    &\leq  \mathbb{E}\left[\frac{B_Y}{k}\cdot|N_{k,n}(X)\triangle N_{k,n+m}(X)|\right],
\end{align*}
since $\max_i |Y_i|\leq B_Y$ holds almost surely by assumption,
where $\triangle$ denotes the symmetric set difference.
Since $|N_{k,n}(X)| = |N_{k,n+m}(X)| = k$ by construction, this means that
$|N_{k,n}(X)\triangle N_{k,n+m}(X)| = 2|N_{k,n+m}(X)\backslash N_{k,n}(X)|$, and so we have
\[\beta^{\textnormal{out}}_{m,n}(\mathcal{A},P)\leq \frac{2B_Y}{k}\cdot \mathbb{E}\left[|N_{k,n+m}(X)\backslash N_{k,n}(X)|\right].\]
Next, by definition of the set of $k$-nearest neighbors, if $i\in[n]$ and $i\in N_{k,n+m}(X)$ then we must have $i\in N_{k,n}(X)$---that is, if $X_i$ is one of the nearest neighbors of $X$ amongst the \emph{larger} collection of points $X_1,\dots,X_{n+m}$, then the same is also true within the smaller collection $X_1,\dots,X_n$. Therefore,
\[\mathbb{E}\left[|N_{k,n+m}(X)\backslash N_{k,n}(X)|\right]
\leq \mathbb{E}\left[|N_{k,n+m}(X)\cap \{n+1,\dots,n+m\}|\right]
= \sum_{i=n+1}^{n+m}\mathbb{P}\{i\in N_{k,n+m}(X)\} = m \cdot \frac{k}{n+m},\]
since, due to exchangeability of the data points, any point is equally likely to be one of the nearest neighbors of $X$. Combining everything, then, we have established that $\beta^{\textnormal{out}}_{m,n}(\mathcal{A},P)\leq 2B_Y \cdot \frac{m}{n+m}$, as desired.

Next we turn to $\beta^{\textnormal{in}}_{m,n}(\mathcal{A},P)$. By an identical calculation, we have
\[\beta^{\textnormal{in}}_{m,n}(\mathcal{A},P)\leq \frac{2B_Y}{k}\cdot \mathbb{E}\left[|N_{k,n+m}(X_1)\backslash N_{k,n}(X_1)|\right]
\leq \frac{2B_Y}{k}\cdot \sum_{i=n+1}^{n+m}\mathbb{P}\{i\in N_{k,n+m}(X_1)\}.\]
Moreover, we must have $1\in N_{k,n+m}(X_1)$ by construction; for the remaining $k-1$ many nearest neighbors, each remaining point (i.e., each $i\in\{2,\dots,n+m\}$) is equally likely to be one of these neighbors, and so $\mathbb{P}\{i\in N_{k,n+m}(X_1)\} = \frac{k-1}{n+m-1} \leq \frac{k}{n+m}$ for each $i=n+1,\dots,n+m$, which leads to the desired bound on $\beta^{\textnormal{in}}_{m,n}(\mathcal{A},P)$.

\subsection{Proof of Proposition~\ref{prop:ridge} (ridge regression)}

Define 
\[\widehat{\Sigma}_n = \frac{1}{n}\sum_{i=1}^n X_iX_i^\top, \ \widehat{S}_n = \frac{1}{n}\sum_{i=1}^n X_iY_i,\]
so that we have
\[\widehat{w}_n = (\widehat{\Sigma}_n + \lambda\mathbf{I}_d)^{-1}\widehat{S}_n.\]
Define also
\[\widehat{\Sigma}_{>n} = \frac{1}{m}\sum_{i=n+1}^{n+m} X_iX_i^\top, \ 
\widehat{S}_{>n} = \frac{1}{m}\sum_{i=n+1}^{n+m} X_iY_i,\]
so that
\[\widehat{w}_{n+m} = \left(\frac{n}{n+m}\widehat{\Sigma}_n + \frac{m}{n+m}\widehat{\Sigma}_{>n} + \lambda\mathbf{I}_d\right)^{-1}\left(\frac{n}{n+m}\widehat{S}_n + \frac{m}{n+m}\widehat{S}_{>n}\right),\]
or equivalently,
\[\widehat{w}_{n+m} = \left(\widehat{\Sigma}_n + \frac{m}{n+m}(\widehat{\Sigma}_{>n}-\widehat{\Sigma}_n) + \lambda\mathbf{I}_d\right)^{-1}\left(\widehat{S}_n + \frac{m}{n+m}(\widehat{S}_{>n}-\widehat{S}_n)\right).\]
Then
\[\|\widehat{w}_{n+m}-\widehat{w}_n\|_2\leq \textnormal{(Term 1)} + \textnormal{(Term 2)},\]
where
\[\textnormal{(Term 1)} = \left\|\left(\widehat{\Sigma}_n + \frac{m}{n+m}(\widehat{\Sigma}_{>n}-\widehat{\Sigma}_n) + \lambda\mathbf{I}_d\right)^{-1} \cdot \frac{m}{n+m}(\widehat{S}_{>n}-\widehat{S}_n)\right\|_2\]
and
\[\textnormal{(Term 2)} = \left\| \left[\left(\widehat{\Sigma}_n + \frac{m}{n+m}(\widehat{\Sigma}_{>n}-\widehat{\Sigma}_n) + \lambda\mathbf{I}_d\right)^{-1} - \left(\widehat{\Sigma}_n  + \lambda\mathbf{I}_d\right)^{-1}\right] \cdot \widehat{S}_n\right\|_2.\]
For (Term 1), since $\widehat{\Sigma}_n + \frac{m}{n+m}(\widehat{\Sigma}_{>n}-\widehat{\Sigma}_n)  = \frac{1}{n+m}\sum_{i=1}^{n+m}X_iX_i^\top\succeq 0$, we have 
\[\textnormal{(Term 1)}\leq \left\|\left(\widehat{\Sigma}_n + \frac{m}{n+m}(\widehat{\Sigma}_{>n}-\widehat{\Sigma}_n) + \lambda\mathbf{I}_d\right)^{-1}\right\| \cdot \frac{m}{n+m}\|\widehat{S}_{>n}-\widehat{S}_n\|_2 \leq \lambda^{-1} \cdot \frac{m}{n+m} \cdot \|\widehat{S}_{>n}-\widehat{S}_n\|_2.\]
Therefore,
\begin{multline*}
   \mathbb{E}\left[\textnormal{(Term 1)}\right]
   \leq \frac{\lambda^{-1}m}{n+m}\mathbb{E}\left[\|\widehat{S}_{>n}-\widehat{S}_n\|_2\right]
   \leq \frac{\lambda^{-1}m}{n+m}\left(\mathbb{E}\left[\|\widehat{S}_{>n}-\mathbb{E}[XY]\|_2\right] + \mathbb{E}\left[\|\widehat{S}_n-\mathbb{E}[XY]\|_2\right]\right)
   \\\leq \frac{\lambda^{-1}m}{n+m}\left(\mathbb{E}\left[\|\widehat{S}_{>n}-\mathbb{E}[XY]\|^2_2\right]^{1/2} + \mathbb{E}\left[\|\widehat{S}_n-\mathbb{E}[XY]\|^2_2\right]^{1/2}\right).
\end{multline*}
We can also calculate
\begin{multline*}
\mathbb{E}\left[\|\widehat{S}_{>n}-\mathbb{E}[XY]\|^2_2\right]
=\mathbb{E}\left[\left\|\frac{1}{m}\sum_{i=n+1}^{n+m}(X_i Y_i - \mathbb{E}[XY])\right\|^2_2\right]\\
=\frac{1}{m^2} \sum_{i=n+1}^{n+m}\mathbb{E}\left[\|X_i Y_i - \mathbb{E}[XY]\|^2_2\right]
\leq \frac{1}{m^2} \sum_{i=n+1}^{n+m} (2B_XB_Y)^2
=\frac{4B_X^2B_Y^2}{m},
\end{multline*}
where the second step holds since the data points are independent.
We can also verify that $\mathbb{E}\left[\|\widehat{S}_n-\mathbb{E}[XY]\|^2_2\right]\leq \frac{4B_X^2B_Y^2}{n}$, by a similar calculation. Combining these calculations, then,
\[\mathbb{E}\left[\textnormal{(Term 1)}\right]
\leq  \frac{2B_XB_Y}{\lambda} \cdot \frac{m}{n+m} \cdot\left(\frac{1}{\sqrt{m}}+\frac{1}{\sqrt{n}}\right).\]

Next we turn to (Term 2). 
First, by considering the derivative of the matrix inverse, we have
\[\left(\widehat{\Sigma}_n + \frac{m}{n+m}(\widehat{\Sigma}_{>n}-\widehat{\Sigma}_n) + \lambda\mathbf{I}_d\right)^{-1} - \left(\widehat{\Sigma}_n  + \lambda\mathbf{I}_d\right)^{-1}
=- \left(\widetilde\Sigma_t + \lambda\mathbf{I}_d\right)^{-1} \cdot \left[\frac{m}{n+m}(\widehat{\Sigma}_{>n}-\widehat{\Sigma}_n)\right] \cdot \left(\widetilde\Sigma_t+ \lambda\mathbf{I}_d\right)^{-1}\]
for some $t\in[0,1]$, 
where \[\widetilde\Sigma_t = \widehat{\Sigma}_n + t\cdot \frac{m}{n+m}(\widehat{\Sigma}_{>n}-\widehat{\Sigma}_n).\] We can rewrite
\[\widetilde\Sigma_t = \widehat{\Sigma}_n \cdot \frac{n + (1-t)m}{n+m} + \widehat{\Sigma}_{>n} \cdot \frac{tm}{n+m} \succeq 0,\]
which implies $\|(\widetilde\Sigma_t+ \lambda\mathbf{I}_d)^{-1}\|\leq \lambda^{-1}$, and therefore,
\[\left\|\left(\widehat{\Sigma}_n + \frac{m}{n+m}(\widehat{\Sigma}_{>n}-\widehat{\Sigma}_n) + \lambda\mathbf{I}_d\right)^{-1} - \left(\widehat{\Sigma}_n  + \lambda\mathbf{I}_d\right)^{-1}\right\| \leq \lambda^{-2} \cdot \frac{m}{n+m}\cdot\|\widehat{\Sigma}_{>n} - \widehat{\Sigma}_n\|.\]
Combining everything so far, then, we have
\[\textnormal{(Term 2)}\leq  \lambda^{-2} \cdot \frac{m}{n+m}\cdot\|\widehat{\Sigma}_{>n} - \widehat{\Sigma}_n\| \cdot \|\widehat{S}_n\|_2.\]
Since $\|\widehat{S}_n\|_2\leq B_XB_Y$ by construction, we can therefore bound
\[\mathbb{E}\left[\textnormal{(Term 2)}\right]\leq  \frac{B_XB_Y}{\lambda^2} \cdot \frac{m}{n+m}\cdot\mathbb{E}\left[\|\widehat{\Sigma}_{>n} - \widehat{\Sigma}_n\|\right].\]
We also have
\begin{align*}
    \mathbb{E}\left[\|\widehat{\Sigma}_{>n} - \widehat{\Sigma}_n\|\right]
&\leq \mathbb{E}\left[\|\widehat{\Sigma}_{>n} - \mathbb{E}[XX^\top]\|\right] +\mathbb{E}\left[\|\widehat{\Sigma}_n - \mathbb{E}[XX^\top]\|\right] \\
&\leq \mathbb{E}\left[\|\widehat{\Sigma}_{>n} - \mathbb{E}[XX^\top]\|_{\textnormal{F}}\right]+\mathbb{E}\left[\|\widehat{\Sigma}_n - \mathbb{E}[XX^\top]\|_{\textnormal{F}}\right] \\
&\leq \mathbb{E}\left[\|\widehat{\Sigma}_{>n} - \mathbb{E}[XX^\top]\|^2_{\textnormal{F}}\right]^{1/2} +\mathbb{E}\left[\|\widehat{\Sigma}_n - \mathbb{E}[XX^\top]\|^2_{\textnormal{F}}\right]^{1/2}
\end{align*}
since the matrix operator norm is bounded by the Frobenius norm. We can then calculate
\begin{align*}
    \mathbb{E}\left[\|\widehat{\Sigma}_{>n} - \mathbb{E}[XX^\top]\|^2_{\textnormal{F}}\right]
    &=\mathbb{E}\left[\left\|\frac{1}{m}\sum_{i=n+1}^{n+m}(X_iX_i^\top - \mathbb{E}[XX^\top])\right\|^2_{\textnormal{F}}\right]\\
    &=\frac{1}{m^2}\sum_{i=n+1}^{n+m}\mathbb{E}\left[\|X_iX_i^\top - \mathbb{E}[XX^\top]\|^2_{\textnormal{F}}\right]\\
    &\leq \frac{1}{m^2}\sum_{i=n+1}^{n+m} (2B_X^2)^2 = \frac{4B_X^4}{m},
\end{align*}
where the second step holds since the data points are independent. And, by a similar calculation, we can also verify that $\mathbb{E}\left[\|\widehat{\Sigma}_n - \mathbb{E}[XX^\top]\|^2_{\textnormal{F}}\right]\leq \frac{4B_X^4}{n}$. Combining everything, then,
\[\mathbb{E}\left[\textnormal{(Term 2)}\right]\leq  \frac{2B_X^3B_Y}{\lambda^2} \cdot \frac{m}{n+m}\cdot\left(\frac{1}{\sqrt{m}} + \frac{1}{\sqrt{n}}\right).\]
Finally, combining (Term 1) and (Term 2),
\[\mathbb{E}\left[\|\widehat{w}_{n+m}-\widehat{w}_n\|_2\right]\leq \frac{2B_XB_Y}{\lambda}\left(1 + \frac{B_X^2}{\lambda}\right) \cdot \frac{m}{n+m} \cdot\left(\frac{1}{\sqrt{m}}+\frac{1}{\sqrt{n}}\right).\]

We therefore have
\begin{multline*}\beta^{\textnormal{out}}_{m,n}(\mathcal{A},P) = \mathbb{E}\left[\left|\widehat\mu_{n+m}(X) - \widehat\mu_n(X)\right|\right] 
= \mathbb{E}\left[ \left|\widehat{w}_{n+m}^\top X - \widehat{w}_n^\top X \right|\right]\\
\leq \mathbb{E}\left[\|X\|_2 \cdot \|\widehat{w}_{n+m}-\widehat{w}_n\|_2\right] \leq B_X \cdot \mathbb{E}\left[\|\widehat{w}_{n+m}-\widehat{w}_n\|_2\right].
\end{multline*}
Similarly, it also holds that
\[\beta^{\textnormal{in}}_{m,n}(\mathcal{A},P) = \mathbb{E}\left[\left|\widehat\mu_{n+m}(X_1) - \widehat\mu_n(X_1)\right|\right]  \leq B_X \cdot \mathbb{E}\left[\|\widehat{w}_{n+m}-\widehat{w}_n\|_2\right].\]
Plugging in our bound on $\mathbb{E}\left[\|\widehat{w}_{n+m}-\widehat{w}_n\|_2\right]$ thus completes the proof.

\subsection{Proof of Proposition~\ref{prop:bagging} (bagging)}

First, following the terminology of \citet{soloff2023bagging,soloff2024stability}, we define ``derandomized'' bagging by taking $B\to\infty$---in other words, bagging averages over $B$ many randomly sampled bags, while derandomized bagging takes an expected value over the random bag (i.e., the random subset of $N$ out of $n$ data points, sampled uniformly without replacement). Formally, given a data set $\mathcal{D}_n=  \big\{(X_i,Y_i)\big\}_{i\in[n]}$, derandomized bagging returns the model
\[\widetilde{\mathcal{A}}_{\textnormal{bag}}(\mathcal{D}_{n}) = \widetilde\mu_{n}\textnormal{ where }\widetilde\mu_{n}(x) = \frac{1}{{n\choose N}}\sum_{S\subseteq[n], |S|=N} \widehat\mu_S(x),\]
where $\widehat\mu_S = \mathcal{A}(\{(X_i,Y_i)\}_{i\in S})$ is the model trained on the subsample defined by indices $i\in S$. 

\citet[Theorem 6]{soloff2024stability} implies that
\[\mathbb{E}\left[\left(\widetilde{\mu}_{n}(X) - \widetilde{\mu}_{n+1}(X)\right)^2\right]\leq \frac{1}{4n}\cdot \frac{N}{n+1-N},\]
for any $n$. We can view this as a stronger version of the out-of-sample $1$-stability condition---that is, this bounds the expected \emph{squared} perturbation, i.e., the $L_2$ norm rather than the $L_1$ norm.
Replacing $n$ with $n+k-1$ (for some $1\leq k\leq m$), then,
\[\mathbb{E}\left[\left(\widetilde{\mu}_{n+k-1}(X) - \widetilde{\mu}_{n+k}(X)\right)^2\right]\leq \frac{1}{4(n+k-1)}\cdot \frac{N}{n+k-N} .\]
Here, and below, for each $k\geq 0$ we define the model $\widetilde{\mu}_{n+k} = \widetilde{\mathcal{A}}_{\textnormal{bag}}(\mathcal{D}_{n+k})$, i.e., the output of derandomized bagging run on the data set $\mathcal{D}_{n+k} = \big\{(X_i,Y_i)\big\}_{i\in[n+k]}$, where $(X_1,Y_1),(X_2,Y_2),\dots$ denotes a sequence of data points sampled i.i.d.\ from $P$.

We then have
\begin{multline*}
    \mathbb{E}\left[\left(\widetilde{\mu}_n(X) - \widetilde{\mu}_{n+m}(X)\right)^2\right]
     = \mathbb{E}\left[\left(\sum_{k\in [m]}  \widetilde{\mu}_{n+k-1}(X) - \widetilde{\mu}_{n+k}(X)\right)^2\right]\\
    =\sum_{k\in [m]} \mathbb{E}\left[\left(\widetilde{\mu}_{n+k-1}(X) - \widetilde{\mu}_{n+k}(X)\right)^2\right]  + \sum_{k\neq k'\in[m]}\mathbb{E}\left[ \left(\widetilde{\mu}_{n+k-1}(X) - \widetilde{\mu}_{n+k}(X)\right)\cdot \left(\widetilde{\mu}_{n+k'-1}(X) - \widetilde{\mu}_{n+k'}(X)\right)\right].
\end{multline*}
For the first term, we have
\[\sum_{k\in [m]} \mathbb{E}\left[\left(\widetilde{\mu}_{n+k-1}(X) - \widetilde{\mu}_{n+k}(X)\right)^2\right] \leq \sum_{k\in[m]} \frac{1}{4(n+k-1)}\cdot \frac{N}{n+k-N} \leq \frac{m}{4n}\cdot \frac{N}{n+1-N}, \]
from the calculations above. Next we will verify that the second term is zero. Assume $k<k'$ without loss of generality. We need to verify that
\[\mathbb{E}\left[ \left(\widetilde{\mu}_{n+k-1}(X) - \widetilde{\mu}_{n+k}(X)\right)\cdot \left(\widetilde{\mu}_{n+k'-1}(X) - \widetilde{\mu}_{n+k'}(X)\right)\right]=0.\]
First, for any $i\in[n+k]$, define 
\[\widetilde{\mu}_{[n+k]\backslash i} = \widetilde{\mathcal{A}}_{\textnormal{bag}}(\mathcal{D}_{[n+k]\backslash i})\textnormal{ where }\mathcal{D}_{[n+k]\backslash i} = \big\{(X_j,Y_j)\big\}_{j\in [n+k]\backslash i} ,\]
the model fitted to the data set $\mathcal{D}_{n+k}$ with data point $(X_i,Y_i)$ removed. Then we must have
\begin{multline*}\mathbb{E}\left[ \left(\widetilde{\mu}_{n+k-1}(X) - \widetilde{\mu}_{n+k}(X)\right)\cdot \left(\widetilde{\mu}_{n+k'-1}(X) - \widetilde{\mu}_{n+k'}(X)\right)\right]\\
= \mathbb{E}\left[ \left(\widetilde{\mu}_{[n+k]\backslash i}(X) - \widetilde{\mu}_{n+k}(X)\right)\cdot \left(\widetilde{\mu}_{n+k'-1}(X) - \widetilde{\mu}_{n+k'}(X)\right)\right],\end{multline*}
by symmetry (specifically, because the data points are i.i.d., and the regression algorithm treats data points symmetrically---and therefore, the expected value does not change if we swap the $i$th and the $(n+k)$th data points). Therefore, taking an average over all $i$, we have
\begin{multline*}\mathbb{E}\left[ \left(\widetilde{\mu}_{n+k-1}(X) - \widetilde{\mu}_{n+k}(X)\right)\cdot \left(\widetilde{\mu}_{n+k'-1}(X) - \widetilde{\mu}_{n+k'}(X)\right)\right]\\
= \frac{1}{n+k}\sum_{i=1}^{n+k}\mathbb{E}\left[ \left(\widetilde{\mu}_{[n+k]\backslash i}(X) - \widetilde{\mu}_{n+k}(X)\right)\cdot \left(\widetilde{\mu}_{n+k'-1}(X) - \widetilde{\mu}_{n+k'}(X)\right)\right]\\
= \mathbb{E}\left[ \left( \frac{1}{n+k}\sum_{i=1}^{n+k}\widetilde{\mu}_{[n+k]\backslash i}(X) - \widetilde{\mu}_{n+k}(X)\right)\cdot \left(\widetilde{\mu}_{n+k'-1}(X) - \widetilde{\mu}_{n+k'}(X)\right)\right].
\end{multline*}
Next we have
\begin{multline*}
    \frac{1}{n+k}\sum_{i=1}^{n+k}\widetilde{\mu}_{[n+k]\backslash i}(X) = \frac{1}{n+k}\sum_{i=1}^{n+k} \frac{1}{{n+k-1\choose N}}\sum_{S\subseteq [n+k]\backslash i, |S|=N} \widehat\mu_S(X)\\
= \frac{1}{{n+k\choose N}}\sum_{S\subseteq[n+k], |S|=N} \widehat\mu_S(X) = \widetilde{\mu}_{n+k}(X).
\end{multline*}
In other words, the average $\frac{1}{n+k}\sum_{i=1}^{n+k}\widetilde{\mu}_{[n+k]\backslash i}(X)$ is constructed by randomly sampling $i\in[n+k]$, and then randomly sampling a subset $S\subseteq[n+k]\backslash i$; this is equivalent to simply sampling $S\subseteq [n+k]$, i.e., computing $\widetilde{\mu}_{n+k}(X)$.
Therefore, \[\left( \frac{1}{n+k}\sum_{i=1}^{n+k}\widetilde{\mu}_{[n+k]\backslash i}(X) - \widetilde{\mu}_{n+k}(X)\right)=0\] almost surely. This verifies that $\mathbb{E}\left[ \left(\widetilde{\mu}_{n+k-1}(X) - \widetilde{\mu}_{n+k}(X)\right)\cdot \left(\widetilde{\mu}_{n+k'-1}(X) - \widetilde{\mu}_{n+k'}(X)\right)\right]=0$ for any $k\neq k'$.

Combining our calculations, we have the bound
\[\mathbb{E}\left[\left|\widetilde{\mu}_n(X) - \widetilde{\mu}_{n+m}(X)\right|\right]
\leq \mathbb{E}\left[\left(\widetilde{\mu}_n(X) - \widetilde{\mu}_{n+m}(X)\right)^2\right]^{1/2} \leq \sqrt{\frac{m}{4n}\cdot \frac{N}{n+1-N}}.\]
In other words, we have established an $m$-stability bound for derandomized bagging, i.e., $\widetilde{\mathcal{A}}_{\textnormal{bag}}$. 

Now we return to the setting of finitely many bags. Writing $\widehat\mu_n = \mathcal{A}_{\textnormal{bag}}(\mathcal{D}_n)$ and $\widehat\mu_{n+m} = \mathcal{A}_{\textnormal{bag}}(\mathcal{D}_{n+m})$ (where as before, $\mathcal{A}_{\textnormal{bag}}$ denotes averaging over $B$ many randomly sampled bags), we have
\begin{multline*}\beta^{\textnormal{out}}_{m,n}(\mathcal{A}_{\textnormal{bag}},P)
= \mathbb{E}\left[\left|\widehat\mu_n(X) - \widehat\mu_{n+m}(X)\right|\right]
\\\leq \mathbb{E}\left[\left|\widetilde{\mu}_n(X) - \widetilde{\mu}_{n+m}(X)\right|\right]
+ \mathbb{E}\left[\left|\widehat{\mu}_{n+m}(X) - \widetilde{\mu}_{n+m}(X)\right|\right] + \mathbb{E}\left[\left|\widehat{\mu}_n(X) - \widetilde{\mu}_n(X)\right|\right]\\
\leq \sqrt{\frac{m}{4n} \cdot \frac{N}{n+1-N}}+ \mathbb{E}\left[\left|\widehat{\mu}_{n+m}(X) - \widetilde{\mu}_{n+m}(X)\right|\right] + \mathbb{E}\left[\left|\widehat{\mu}_n(X) - \widetilde{\mu}_n(X)\right|\right].
\end{multline*}
Next, conditional on $X$ and on $\mathcal{D}_{n+m}$, the prediction $\widehat{\mu}_{n+m}(X)$ is the sample mean of $n+m$ many i.i.d.\ draws from a distribution taking values in $[0,1]$, and with mean $\widetilde{\mu}_{n+m}(X)$---that is, bagging (with $B$ many bags) is simply an empirical approximation to derandomized bagging. Therefore,  $\mathbb{E}\left[\left|\widehat{\mu}_{n+m}(X) - \widetilde{\mu}_{n+m}(X)\right|\right]\leq \frac{1}{2\sqrt{B}}$. The same bound holds for $\mathbb{E}\left[\left|\widehat{\mu}_n(X) - \widetilde{\mu}_n(X)\right|\right]$. Combining everything, then, we have shown that
\[\beta^{\textnormal{out}}_{m,n}(\mathcal{A}_{\textnormal{bag}},P)\leq \sqrt{\frac{m}{4n} \cdot \frac{N}{n+1-N}} + \frac{1}{\sqrt{B}}.\]

\end{document}